\documentclass[reqno, 10pt, letterpaper]{amsart}
\usepackage{amsfonts,amsmath, amsthm, amssymb, latexsym, epsfig}

\usepackage[all]{xy}
\usepackage{xspace}
\usepackage{comment}
\usepackage{setspace}
\usepackage{enumerate}
\usepackage[mathscr]{eucal}

\usepackage{bbold}

\usepackage{color}
\usepackage{hyperref}

        \advance \topmargin by -\headheight
        \advance \topmargin by -\headsep

        \evensidemargin \oddsidemargin
	\newcommand{\identicalto}[1]{\footnote{Identical to #1}}
	\newcommand{\partof}[1]{\footnote{Part of #1}}
	\newcommand{\similarto}[1]{\footnote{Similar to #1}}

        \newcommand{\eg}{\emph{e.g.}\xspace} 

        \newcommand{\Io}[1][n]{\ensuremath{\mathbb{I}_{#1,0}}\xspace}

        \newcommand{\e}[2]{\ensuremath{\mathfrak{e}_{#1,#2}}\xspace}

        \newcommand{\F}[2]{\ensuremath{\mathfrak{f}_{#1,#2}}\xspace}

	\newcommand{\ftn}[3]{ #1 \colon #2 \rightarrow #3 }
	\newcommand{\setof}[2]{\ensuremath{\left\{ #1 \: : \: #2 \right\}}}
        \newcommand{\ext}[3]{\ensuremath{#1\hookrightarrow#2\twoheadrightarrow#3}\xspace}

	\newcommand{\A}{\ensuremath{\mathfrak{A}}\xspace}
        \newcommand{\B}{\ensuremath{\mathfrak{B}}\xspace}

        \newcommand{\Homsix}{\ensuremath{\operatorname{Hom}_{\mathrm{six}}}\xspace}

	\newcommand{\kk}{\ensuremath{\mathit{KK}}\xspace}

	\newcommand{\kkE}{\ensuremath{\mathit{KK}_{\mathcal{E}}}\xspace}

	\newcommand{\Hom}{\ensuremath{\operatorname{Hom}}}
	
	\newcommand{\Aut}{\ensuremath{\operatorname{Aut}}}
\newcommand{\CInnO}{\ensuremath{\operatorname{\overline{Inn}_0}}}
\newcommand{\CInn}{\ensuremath{\operatorname{\overline{Inn}}}}

	\newcommand{\diag}{\ensuremath{\operatorname{diag}}}
	\newcommand{\im}{\ensuremath{\operatorname{im}}}
	\newcommand{\id}{\ensuremath{\operatorname{id}}}

	\newcommand{\multialg}[1]{\mathcal{M}(#1)\xspace}

	\newcommand{\starhoms}{{$*$}-homomorphisms\xspace}
	
	\newcommand{\starautos}{{$*$}-automorphisms\xspace}

	\newcommand{\cstar}{{$C \sp \ast$}\xspace}

        \newcommand{\kE}{\ensuremath{\underline{\mathit{K}}_{\mathcal{E}}}\xspace}

	\newcommand{\Z}{\ensuremath{\mathbb{Z}}\xspace}
	\newcommand{\C}{\ensuremath{\mathbb{C}}\xspace}

	\newcommand{\N}{\ensuremath{\mathbb{N}}\xspace}
	
	\newcommand{\K}{\ensuremath{\mathbb{K}}\xspace}
	\newcommand{\ksix}{\ensuremath{K_{\mathrm{six}}}\xspace}

        \newcommand{\cf}{\emph{cf.}\xspace} 
        \newcommand{\ie}{\emph{i.e.}\xspace}

	\theoremstyle{plain}
	\newtheorem{thm}{Theorem}[section]
	\newtheorem{lemma}[thm]{Lemma}
	\newtheorem{theor}[thm]{Theorem}
	\newtheorem{propo}[thm]{Proposition}
	\newtheorem{corol}[thm]{Corollary}

	\newtheorem{assumption}[thm]{Assumption}
	\newtheorem{remar}[thm]{Remark}	
	\theoremstyle{definition}
	\newtheorem{defin}[thm]{Definition}

	\numberwithin{equation}{section}
	\numberwithin{figure}{section}

	\newcommand{\bolddefine}[1]{\textbf{#1}}

\begin{document}
	\title{Automorphisms of Cuntz-Krieger algebras}
	\author{S{\o}ren Eilers}
        \address{Department of Mathematical Sciences \\
        University of Copenhagen\\
        Universitetsparken~5 \\
        DK-2100 Copenhagen, Denmark}
        \email{eilers@math.ku.dk }
        \author{Gunnar Restorff}
        \address{Department of Science and Technology \\ 
        University of the Faroe Islands \\
	N\'oat\'un~3 \\ 
	FO-100 T\'orshavn, Faroe Islands}
	\email{gunnarr@setur.fo}
	\author{Efren Ruiz}
        \address{Department of Mathematics\\University of Hawaii,
	Hilo\\200 W. Kawili St.\\
	Hilo, Hawaii\\
	96720-4091 USA}
        \email{ruize@hawaii.edu}
        \date{\today}



	\begin{abstract}
	We prove that the natural homomorphism from Kirchberg's
        ideal-related $\kk$-theory, $\kkE( e, e' )$, with one
        specified ideal, into $\mathrm{Hom}_{ \Lambda } (
        \underline{K}_{\mathcal{E}} ( e ) ,
        \underline{K}_{\mathcal{E}} ( e' ) )$ is an isomorphism for
        all extensions $e$ and $e'$ of separable, nuclear
        $C^{*}$-algebras in the bootstrap category $\mathcal{N}$ with
        the $K$-groups of the associated cyclic six term exact sequence being finitely generated, 
        having zero exponential map and with
        the $K_{1}$-groups of the quotients being free abelian groups. 

	This
        class includes all Cuntz-Krieger algebras with
        exactly one non-trivial ideal. 
        Combining our results with the results of Kirchberg, we classify automorphisms of the stabilized purely infinite Cuntz-Krieger algebras with exactly one non-trivial ideal modulo asymptotically unitary equivalence.  We also get a classification result modulo approximately unitary equivalence.

The results in this paper also apply to certain graph algebras.
	\end{abstract}
       	\maketitle

Understanding groups of \starautos has been a key ambition since the early days of \cstar-algebra theory as perhaps most clearly indicated by the title of \cite{gkp:csaatag}. The subgroups of inner automorphisms and inner automorphisms given by unitaries in the base component, respectively, have closures
$$\CInn(\A) \qquad \CInnO(\A)$$
which are normal subgroups, so one usually focuses attention on the quotient groups
$$\Aut(\A)/\CInn(\A)\qquad \Aut(\A)/\CInnO(\A),$$
trying in particular to describe the latter groups by means of $K$-theory in situations where the \cstar-algebra in question has already been seen to be classifiable using a certain $K$-theoretical invariant. Such analyses have already been carried out, resulting in very satisfactory results, in many cases where the \cstar-algebra in question is either stably finite or simple, cf.\ \cite{kirchpure}, \cite{ncp:actfnpiscsa}, \cite{gaemr:tagotircsa}, \cite{multcoeff}, \cite{pwner:tagoastaa}, and~\cite{pwner:tagozs}.

In the present paper we consider the group $\Aut(\A)/\CInn(\A)$ for a selected class of \cstar-algebras which are non-simple and purely infinite. We focus on the most basic class of non-simple Cuntz-Krieger algebras, those in relation to dynamical systems which are \emph{two-component} in the sense of Huang, \cite{dh:tcotccka}, and hence are given with a unique ideal. These \cstar-algebras are classified up to stable isomorphism by the cyclic six term exact sequence in $K$-theory, but although three independent proofs of this fact have been given --- by R\o{}rdam (\cite{extpurelyinf}), by Bonkat (\cite{bonkat}), and by the second named author (\cite{cp:ckalg}) --- none of these allow the computation of $\Aut( \A)/\CInn(\A)$. Furthermore, it was proved in \cite{err:nsikirkkt} that the most obvious $K$-theoretical candidate for an invariant which could be used to compute  $\Aut(\A)/\CInn(\A)$ fails in this case.

Inspired by the work of Kirchberg-Phillips and Dadarlat-Loring in the simple case and Bonkat, Meyer-Nest and Kirchberg in the non-simple case, we use the profound isomorphism result by Kirchberg (\cite{kirchpure}) to almost reduce the problem to computing a certain group in Kirchberg's equivariant $\kk$-theory. As in the simple case with Rosenberg-Schochet's Universal Coefficient Theorem (\cite{jrcs:tktatuctfkgkf}), Bonkat's UCT (\cite{bonkat}) is not of much help here, because the existence of such a UCT does not, in itself, aid us in computing $\Aut(\A)/\CInn(\A)$, since it mainly gives us information about the additive structure of $\kk$-groups and not the multiplicative structure (moreover, in contrast to the usual UCT, Bonkat's UCT does not even split, cf.\ \cite{err:nsikirkkt}). 
In the simple case, the problem was already solved by Dadarlat-Loring by having proved a Universal Multi-Coefficient Theorem (UMCT), cf.\ \cite{multcoeff}. This UMCT has turned out to be very important to a large variety of problems, including determining automorphism groups, extending known and proving new classification theorems for \cstar-algebras, providing a large variety of examples and counterexamples, as well as determining $\kk$ as a ring. 
It turns out, that the tools developed in \cite{err:idealkthy} may be refined to yield a complete algebraic description of Kirchberg's ideal-related $\kk$-group in terms of a UMCT in the case of interest. This then combines with recent results on establishing semiprojectivity of certain \cstar-algebras with one distinguished ideal to give the desired description of $\Aut(\A)/\CInn(\A)$. 

The core of our argument is a rather complex algebraic analysis of the proposed invariant which draws on the fact that Cuntz-Krieger algebras have real rank zero, and that the $K_1$-groups of Cuntz-Krieger algebras are always free and finitely generated. Although it is clear that the present UMCT will hold for a class that is much larger than the class of Cuntz-Krieger algebras (see Remark~\ref{rem-remaboutscopeofinvariant}), it is not easy to venture a guess about whether or not the invariant which allows our description of the automorphism groups of Cuntz-Krieger algebras could be amended to all cases of purely infinite \cstar-algebras with a unique ideal. 

The paper is organized as follows. 
In Section~1, the definitions and results from \cite{err:idealkthy} are recalled --- among there the definition of ideal-related $K$-theory with coefficients. 
In Section~2, some concrete, very specific projective resolutions of extensions of Kirchberg algebras are constructed. 
These are important for the proof of the main theorem. 
In Section~3, we prove exactness of a certain sequence, which serves as the main technical ingredient of the proof of the main theorem. 
In Section~4, the main theorem, Theorem~\ref{thm:main}, is given, stating a UMCT for the algebras in question. 
In Section~5, a reduced invariant for Cuntz-Krieger algebras is introduced, while in Section~6, the main theorem is used to determine the above mentioned groups $\Aut(\A)/\CInn(\A)$ for two-component Cuntz-Krieger algebras.  

\section{Definition of $\underline{K}_{\mathcal{E}}$}

In this section we recall some definitions and results from \cite{err:idealkthy} -- specifically we recall the definition of the ideal-related $K$-theory with coefficients, $\underline{K}_{\mathcal{E} } (e)$. Moreover we prove a few lemmata which are needed later in this paper. 
Throughout the paper $\N$, $\N_0$ and $\N_{\geq 2}$ will denote the positive integers, the nonnegative integers, and the integers greater or equal to two, respectively. Moreover, $\mathsf{M}_n$ will denote the $n\times n$ matrices with entries from the complex numbers $\C$. 

\begin{defin}\identicalto{\cite[Def.~2.1]{err:idealkthy}}
Let $\mathfrak{A}$ be a $C^{*}$-algebra.  Then we define the \bolddefine{suspension} and the \bolddefine{cone} of $\mathfrak{A}$ as 
\begin{align*}
\mathsf{S}\mathfrak{A} = \setof{ f \in C( [0, 1] , \mathfrak{A} ) }{ f(0)=0, f(1) = 0 } \qquad
\mathsf{C}\mathfrak{A} = \setof{ f \in C( [0, 1 ] , \mathfrak{A} ) }{ f(0) = 0 }
\end{align*}
respectively.
\end{defin}

\begin{defin}\identicalto{\cite[Def.~6.1]{err:idealkthy}}
Let $n \in \N_{\geq 2}$.  We let $\mathbb{I}_{n,0}$ denote the (non-unital) \bolddefine{dimension-drop interval}, \ie, $\mathbb{I}_{n,0}$ is the mapping cone of the unital $*$-homomorphism from $\C$ to $\mathsf{M}_{n}$.
\end{defin}
\begin{defin}\partof{\cite[Def.~2.7]{err:idealkthy}}
For an extension $e \colon \mathfrak{A}_{0} \overset{ \iota }{ \hookrightarrow } \mathfrak{A}_{1} \overset{ \pi }{ \twoheadrightarrow } \mathfrak{A}_{2}$, we let
\begin{align*}
\mathfrak{mc}(e) \colon \mathsf{S} \mathfrak{A} _{2}\overset{ \iota_{\mathfrak{mc} } }{ \hookrightarrow } \mathsf{C}_{\pi} \overset{ \pi_{ \mathfrak{mc} } }{ \twoheadrightarrow } \mathfrak{A}_{1} \\
\mathsf{S}e \colon \mathsf{S} \mathfrak{A}_{0} \overset{ \mathsf{S} \iota }{ \hookrightarrow } \mathsf{S} \mathfrak{A}_{1} \overset{ \mathsf{S} \pi }{ \twoheadrightarrow } \mathsf{S} \mathfrak{A}_{2}
\end{align*}
be the mapping cone sequence and suspension sequence of $e$, respectively.
\end{defin}

\begin{defin}\identicalto{\cite[Def.~6.2]{err:idealkthy}}
Let $n \in \N_{\geq 2}$.  We let $\mathfrak{e}_{n,0}$ denote the mapping cone sequence 
\begin{equation*}
\mathfrak{e}_{n,0} \colon \mathsf{S}\mathsf{M}_{n} \hookrightarrow \mathbb{I}_{n,0} \twoheadrightarrow \C
\end{equation*}
corresponding to the unital $*$-homomorphism from $\C$ to $\mathsf{M}_{n}$.  We let, moreover, $\mathfrak{e}_{n,i} = \mathfrak{mc}^{i} ( \mathfrak{e}_{n,0} )$, for all $i\in\N$, and we write
\begin{center}
$\mathfrak{e}_{n,1} \colon \mathsf{S} \C \hookrightarrow \mathbb{I}_{n,1} \twoheadrightarrow \mathbb{I}_{n,0}$\\
$\mathfrak{e}_{n, i } \colon \mathsf{S} \mathbb{I}_{n,i-2} \hookrightarrow \mathbb{I}_{n, i} \twoheadrightarrow \mathbb{I}_{n, i-1}, \ \text{for $i \geq 2$}.$ 
\end{center}

Similarly, we set $\mathfrak{f}_{1,0} \colon  \C \overset{ \id }{ \hookrightarrow } \C \twoheadrightarrow 0$ and $\mathfrak{f}_{n,0} \colon \mathbb{I}_{n,0} \overset{\id}{ \hookrightarrow } \mathbb{I}_{n,0} \twoheadrightarrow 0$, for all $n \in \N_{ \geq 2}$.  Moreover, we set $\mathfrak{f}_{n,i} = \mathfrak{mc}^{i} ( \mathfrak{f}_{n,0} )$ for all $n \in \N$ and all $i \in \N$.
	\end{defin}

\begin{defin}\identicalto{\cite[Def.~1.4]{err:idealkthy} --- although $3,4,5$ is missing there}
For each extension $e$ of separable $C^{*}$-algebras, we define \bolddefine{ideal-related $K$-theory with coefficients}, $\underline{K}_{\mathcal{E} } (e)$, of $e$ to be the (graded) group
\begin{equation*}
\underline{K}_{\mathcal{E} } ( e) = \bigoplus_{ i = 0 }^{5} \left( \kkE( \mathfrak{f}_{1, i } , e ) \oplus \bigoplus_{ n = 2}^{ \infty } \kkE ( \mathfrak{e}_{n,i}, e) \oplus \kkE ( \mathfrak{f}_{n,i} , e )  \right).
\end{equation*}
A homomorphism $\alpha$ from $\kE ( e_{1} )$ to $\kE ( e_{2} )$ is a group homomorphism respecting the direct sum decomposition and the natural homomorphisms induced by the elements $\kkE ( e, e' )$, where $e$ and $e'$ are in $\setof{ \mathfrak{e}_{n,i} , \mathfrak{f}_{n,i},\mathfrak{f}_{1,i} }{ n \in \N_{ \geq 2}, i = 0,1,2,3,4,5}$.  The set of homomorphisms from $\kE( e_{1} )$ to $\kE ( e_{2} )$ will be denoted by $\Hom_{\Lambda } ( \kE ( e_{1} ), \kE ( e_{2} ) )$.

Let $x \in \kk_{ \mathcal{E} } ( e_{1} , e_{2} )$.  Then $x$ induces an element of $\Hom_{\Lambda} ( \kE ( e_{1} ), \kE ( e_{2} ) )$ by 
\begin{align*}
&y \in \kkE ( \mathfrak{f}_{n, i} , e_{1} ) \mapsto y \times x \in \kkE ( \mathfrak{f}_{n, i} , e_{2} ),\quad n\in\N, \\
&y \in \kkE ( \mathfrak{e}_{n,i} , e_{1} ) \mapsto y \times x \in \kkE ( \mathfrak{e}_{n,i} , e_{2} ),\quad n\in\N_{\geq 2}.
\end{align*}
Hence, if $\ftn{ \phi }{ e_{1} }{ e_{2} }$ is a homomorphism, then $\phi$ induces an element $\kE ( \phi ) \in \Hom_{\Lambda} ( \kE ( e_{1} ), \kE ( e_{2} ) )$. In this way, $\kE$ becomes a functor on the category of extensions. 
\end{defin}

\begin{lemma}\label{l:naturalmap}
Let $e_{1}$ and $e_{2}$ be extensions of separable $C^{*}$-algebras.  Then there is a natural homomorphism, 
\begin{equation*}
\ftn{ \Gamma_{e_{1}, e_{2} } }{ \kkE ( e_{1} , e_{2} ) }{ \Hom_{ \Lambda } ( \underline{K}_{\mathcal{E} } ( e_{1} ), \underline{K}_{\mathcal{E}} ( e_{2} ) ) }.
\end{equation*}
\end{lemma}

\begin{proof}
Let $x \in \kkE ( e_{1} , e_{2} )$.  Then $\ftn{ (-)\times x }{ \kkE ( \mathfrak{f}_{n,i}, e_{1} ) }{ \kkE ( \mathfrak{f}_{n,i} , e_{2} ) }$ and $\ftn{ (-)\times x}{ \kkE ( \mathfrak{e}_{n,i} , e_{1} ) }{ \kkE ( \mathfrak{e}_{n,i}, e_{2} ) }$ are group homomorphisms which respect the natural homomorphisms induced by the elements $\kkE^{j} ( e, e' )$ for $j=0,1$, where $e$ and $e'$ are in $\setof{ \mathfrak{e}_{n,i}, \mathfrak{f}_{n,i}, \mathfrak{f}_{1,i} }{ n \in \N_{\geq 2}, i = 0,1,2,3,4,5}$.  Hence, $\Gamma_{ e_{1} , e_{2} } ( x ) = (-) \times x \in \Hom_{ \Lambda } ( \underline{K}_{\mathcal{E} } ( e_{1} ), \underline{K}_{\mathcal{E}} ( e_{2} ) )$.  One can check that $\Gamma_{e_{1}, e_{2}}$ is a natural homomorphism.
 \end{proof}

\begin{defin}
Let $\ftn{ \phi }{ e_{1} }{ e_{2} }$ be a homomorphism where $e_{1}$ and $e_{2}$ are extensions of $C^{*}$-algebras.  Note that $\phi$ induces $\kE ( \phi ) \in \Hom_{ \Lambda } ( \kE ( e_{1} ) , \kE ( e_{2} ) )$.  Hence, for every extension $e$ of \cstar-algebras, $\phi$ induces a homomorphism $\ftn{ \phi^{*} }{ \Hom_{ \Lambda } ( \kE ( e_{2} ) , \kE ( e ) ) }{ \Hom_{ \Lambda } ( \kE ( e_{1} ) , \kE ( e ) ) }$ by $\phi^{*} ( \alpha ) = \alpha \circ \kE ( \phi )$. 
\end{defin}

\begin{defin}\similarto{\cite[Def.~6.3]{err:idealkthy}}
For each extension $e$ of separable $C^{*}$-algebras, we define \bolddefine{ideal-related $K$-theory}, $\ksix ( e )$, of $e$ to be the (graded) group
\begin{equation*}
\ksix ( e ) = \bigoplus_{ i = 0 }^{5} \kkE( \mathfrak{f}_{1, i } , e )
\end{equation*}
A homomorphism $\ftn{ \alpha }{ \ksix ( e_{1} ) }{ \ksix( e_{2} ) }$ is a group homomorphism respecting the direct sum decomposition and the natural homomorphisms induced by the elements $\kkE^{j} ( e, e' )$ for $j=0,1$, where $e$ and $e'$ are in $\setof{ \mathfrak{f}_{1,i} }{ i = 0,1,2,3,4,5}$.  The set of homomorphism between $\ksix ( e_{1} )$ to $\ksix ( e_{2} )$ will be denoted by $\Hom_{ \mathrm{six}  } ( \ksix ( e_{1} ) , \ksix ( e_{2} ) )$.
\end{defin}

The definition of $\ksix ( e_{1} )$ is different from the definitions in \cite{bonkat} and \cite{err:idealkthy}, but a computation shows that there are natural isomorpisms between the different definitions.

\begin{remar}\identicalto{\cite[Rem.~6.6]{err:idealkthy}}
For extensions $e_1:\ext{\A_0}{\A_1}{\A_2}$ and
$e_2:\ext{\B_0}{\B_1}{\B_2}$ of separable $C^{*}$-algebras, we have natural homomorphisms 
$G_i\colon\kkE(e_1,e_2)\longrightarrow\kk(\A_i,\B_i)$, for $i=0,1,2$.  As in the proof of \cite[Satz~7.5.6]{bonkat} the obvious diagram 
$$\def\objectstyle{\scriptstyle}
\def\labelstyle{\scriptstyle}
\xymatrix{
  \operatorname{Ext}_{\mathrm{six}}(\ksix(e_1),\ksix(\mathsf{S}e_2))\ar@{^(->}[r]\ar[d]
  & \kkE(e_1,e_2)\ar@{->>}[r]\ar[d]^{G_i}
  & \Homsix(\ksix(e_1),\ksix(e_2))\ar[d]\\
  \operatorname{Ext}(K_0(\A_i),K_1(\B_i))
  \oplus\operatorname{Ext}(K_1(\A_i),K_0(\B_i))
  \ar@{^(->}[r]
  & \kk(\A_i,\B_i)\ar@{->>}[r]
  & \Hom(K_0(\A_i),K_0(\B_i))
  \oplus\Hom(K_1(\A_i),K_1(\B_i))   
  }$$
commutes and is natural in $e_2$, for $i=0,1,2$ 
--- provided that $e_1$ belongs to the UCT class considered
by Bonkat.\end{remar}

\begin{remar}
  The functor \kE contains both the usual $K$-theory and 
  the $\mathrm{mod}-n$ $K$-theory. 
  That is, we can recover the functors $K_0$, $K_1$, $K_0(-;\Z_n)$,
  and $K_1(-;\Z_n)$
  restricted to the ideal, the extension algebra, and the quotient. 

  Using the above diagram, we get canonical
  natural isomorphisms
$$\xymatrix@R=12pt{
    \kkE(\mathfrak{f}_{1,i},e)\longrightarrow\kk(\C,\A_i)\cong K_0(\A_i)&
    \kkE(\mathfrak{f}_{1,i+3},e)\longrightarrow\kk(\mathsf{S}\C,\A_i)\cong K_1(\A_i),\\
    \kkE(\mathfrak{f}_{n,i},e)\longrightarrow\kk(\Io,\A_i)\cong K_0(\A_i;\Z_n)&
    \kkE(\mathfrak{f}_{n,i+3},e)\longrightarrow\kk( \mathsf{S} \Io,\A_i)\cong K_1(\A_i;\Z_n),}$$
  for all $n\in\N_{\geq 2}$, and all $i=0,1,2$.

  Since the Bockstein operations, as defined in \cite{multcoeff}, are defined as
  natural transformations given by Kasparov products, also the total
  $K$-theory with coefficients, $\underline{K}$, is included in the
  invariant. 
\end{remar}

\begin{defin}\similarto{\cite[Def.~7.2]{err:idealkthy} (only first part)}
Set $F_{1,i}^{e} = \kkE ( \mathfrak{f}_{1,i} , e )$, $F_{n,i}^{e} = \kkE ( \mathfrak{f}_{n,i } , e )$, and $H_{n,i}^{e} = \kkE ( \mathfrak{e}_{n,i} , e )$ for all $n \in \N_{\geq 2}$ and for all $i = 0 , 1, 2, 3, 4, 5$.  For convenience, we will identify indices modulo $6$, \ie, we write $F_{n, 6 }^{e} = F_{n,0}^{e}$, $F_{n,7}^{e} = F_{n,1}^{e}$ etc. 

For $\alpha \in \Hom_{ \Lambda } ( \kE ( e_{1} ) , \kE ( e_{2} ) )$, set $\ftn{ \alpha_{ \F{1}{ i}} = \alpha \vert_{ F_{1,i}^{e_{1} } } }{ F_{1,i}^{ e_{1} } }{ F_{1,i}^{e_{2} } }$, $\ftn{ \alpha_{ \F{n}{ i}} = \alpha \vert_{ F_{n,i}^{e_{1} } } }{ F_{n,i}^{ e_{1} } }{ F_{n,i}^{e_{2} } }$, and $\ftn{ \alpha_{ \e{n}{i} }  = \alpha \vert_{ H_{n,i}^{e_{1}} } }{ H_{n,i}^{ e_{1}} }{ H_{n,i}^{e_{2}} }$, for all $n\in\N_{\geq 2}$ and all $i=0,1,2,3,4,5$. For all $n \in \N_{\geq 2}$ and $i=0,1,2,3,4,5$, we have homomorphisms
$$\xymatrix@R=12pt@!C=36pt{ F_{1, i-1}^{e} \ar[r]^{ f_{1,i-1}^{e} } & F_{1,i}^{e} \ar[r]^{f_{1,i}^{e}} & F_{1, i+1}^{e} & 
F_{1,i+1}^{e} \ar[r]^-{ h_{n,i,e}^{1,1, in } } & H_{n,i}^{e} \ar[r]^-{h_{n,i, e}^{ 1,1, out}} & F_{1, i+3}^{e} \\
F_{n, i-1}^{e} \ar[r]^{ f_{n,i-1}^{e} } & F_{n,i}^{e} \ar[r]^{f_{n,i}^{e}} & F_{n, i+1}^{e} & 
F_{n,i}^{e} \ar[r]^-{ h_{n,i,e}^{n,1, in } } & H_{n,i}^{e} \ar[r]^-{h_{n,i, e}^{ n,1, out}} & F_{1, i+2}^{e}   \\
F_{1, i }^{e} \ar[r]^-{\rho_{n,i}^{e} }  & F_{n,i}^{e} \ar[r]^-{\beta_{n,i}^{e}} & F_{1, i+3}^{e} & 
F_{1,i+2}^{e} \ar[r]^-{ h_{n,i,e}^{1,n, in } } & H_{n,i}^{e} \ar[r]^-{h_{n,i, e}^{ 1,n, out}} & F_{n, i+1}^{e} }
$$
as defined in \cite{err:idealkthy}.
\end{defin}

\begin{defin}\identicalto{\cite[Def.~7.13]{err:idealkthy}}
For each $n \in \N$, we set $\tilde{f}_{n, i }^{e} = f_{n,i}^{e}$ for $i = 1, 2, 4, 5$ and $\tilde{f}_{n,i}^{e} = - f_{n,i}^{e}$ for $i = 0, 3$.
\end{defin}

\begin{theor}\label{err:idealkthy}(Theorem 7.14 of \cite{err:idealkthy})
Let  $e$ be an extension of separable $C^{*}$-algebras.  For all $n \in \N$ and for all $i = 0, 1, 2, 3, 4, 5$, 
$$\xymatrix{ F_{n,i-1}^{e} \ar[rr]^-{ f_{n,i-1}^e } & & F_{n,i}^{e} \ar[rr]^-{ f_{n,i}^e } & & F_{n,i+1}^{e} }$$ 
is exact.  For all $n \in \N_{ \geq 2}$ and for all $i = 0,1,2,3,4,5$, 
$$\begin{array}{lcr}
\xymatrix@!C=48pt{
F_{1,i+1}^{e} \ar[r]^{h_{n,i,e}^{1,1, in } } & H_{n,i}^{e} \ar[r]^-{h_{n,i,e}^{1,1,out}} & F_{1, i+3}^{e} \ar[d]_{ n f_{1, i+3}^e } \\
F_{1,i}^{e} \ar[u]_{n f_{1, i}^e } & H_{n,i+3}^{e} \ar[l]^-{ h_{n, i+3, e}^{1,1, out } } & F_{1, i+4}^{e} \ar[l]^-{h_{n,i+3, e}^{1, 1, in} }
}
& 
& \xymatrix@!C=48pt{
F_{n,i}^{e} \ar[r]^-{h_{n,i,e}^{n,1, in } } & H_{n,i}^{e} \ar[r]^-{h_{n,i,e}^{n,1,out}} & F_{1, i+2}^{e} \ar[d]_{ f_{n, i+2}^e \circ \rho_{n, i+2}^{e} } \\
F_{1,i+5}^{e} \ar[u]_{f_{n, i+5}^e \circ \rho_{n, i+5 }^{e} } & H_{n,i+3}^{e} \ar[l]^-{ h_{n, i+3, e}^{n,1, out } } & F_{n, i+3}^{e} \ar[l]^-{h_{n,i+3, e}^{n, 1, in} }
} \\
\vcenter{\xymatrix@!C=48pt{
F_{1,i+2}^{e} \ar[r]^-{h_{n,i,e}^{1,n, in } } & H_{n,i}^{e} \ar[r]^-{h_{n,i,e}^{1,n,out}} & F_{n, i+1}^{e} \ar[d]_{ \beta_{n , i+2 }^{ e } \circ f_{n, i+1}^e  } \\
F_{n,i+4}^{e} \ar[u]_{\beta_{n, i+5 }^{ e }\circ f_{n, i+4}^e  } & H_{n,i+3}^{e} \ar[l]^-{ h_{n, i+3, e}^{1,n, out } } & F_{1, i+5}^{e} \ar[l]^-{h_{n,i+3, e}^{1, n, in} }
}}
 & \text{and} & 
\vcenter{\xymatrix@!C=48pt{F_{1, i}^{e} \ar[r]^-{ \rho_{n,i }^{e} } & F_{n,i}^{e} \ar[r]^{ \beta_{n,i }^{ e } } & F_{1, i+3}^{e} \ar[d]_{ \times n } \\
F_{1,i}^{e} \ar[u]_{ \times n} & F_{n, i+3}^{e} \ar[l]^-{ \beta_{n,i+3 }^{ e}} & F_{1, i+3}^{e} \ar[l]^{ \rho_{n, i+3 }^{e} } 
}}
\end{array}$$ 
are exact, and moreover, all the three diagrams
\begin{equation}
\vcenter{
\xymatrix{
F_{1,i}^{e} \ar[r]^-{ \tilde{f}_{1, i}^{ e } } \ar[d]_{ \rho_{n,i}^{e}} & F_{1, i+1}^{e} \ar[d]_(.4){h_{n, i, e}^{1,1, in} } \ar[rd]^-{ \tilde{f}_{1,i+1}^{e} } \\
F_{n, i}^{e} \ar[r]^-{ h_{n,i, e}^{n,1, in } } \ar[rd]_-{\beta_{n,i }^{ e} }  & H_{n, i}^{e}  \ar[r]_-{h_{n,i, e}^{n,1, out} } \ar[d]^(.6){h_{n,i, e}^{1,1,out} } & F_{1, i+2}^{e} \ar[d]^{\tilde{f}_{1, i+2}^{e} } \\
			& F_{1,i+3}^{e} \ar[r]_{\times n } & F_{1, i+3}^{e}
}
}
\label{d1} \tag{$1$}
\end{equation}
\begin{equation}
\vcenter{
\xymatrix{
F_{1,i+1}^{e} \ar[r]^-{ \times n } \ar[d]_{ \tilde{f}_{1,i+1}^{e}} & F_{1, i+1}^{e} \ar[d]_(.4){h_{n, i, e}^{1,1, in} } \ar[rd]^-{ \rho_{n,i+1}^{e} } \\
F_{1, i+2}^{e} \ar[r]^-{ h_{n,i, e}^{1,n, in } } \ar[rd]_-{\tilde{f}_{1,i+2}^{ e} }  & H_{n, i}^{e}  \ar[r]_-{h_{n,i, e}^{1,n, out} } \ar[d]^(.6){h_{n,i, e}^{1,1,out} } & F_{n, i+1}^{e} \ar[d]^{-\beta_{n, i+1}^{e} } \\
			& F_{1,i+3}^{e} \ar[r]_{ \tilde{f}_{1, i+3}^{e} } & F_{1, i+4}^{e}
}
}
\label{d2} \tag{$2$}
\end{equation}
\begin{equation}
\vcenter{
\xymatrix{
F_{n,i+5}^{e} \ar[r]^-{ \tilde{f}_{n, i+5}^{ e } } \ar[d]_{ - \beta_{n,i+5 }^{ e}} & F_{n, i}^{e} \ar[d]_(.4){h_{n, i, e}^{n,1, in} } \ar[rd]^-{ \tilde{f}_{n,i}^{e} } \\
F_{1, i+2}^{e} \ar[r]^-{ h_{n,i, e}^{1,n, in } } \ar[rd]_-{ \times n }  & H_{n, i}^{e}  \ar[r]_-{h_{n,i, e}^{1,n, out} } \ar[d]^(.6){h_{n,i, e}^{n,1,out} } & F_{n, i+1}^{e} \ar[d]^{\tilde{f}_{n, i+1}^{ e} } \\
			& F_{1,i+2}^{e} \ar[r]_{ \rho_{n,i+2}^{e} } & F_{n, i+2}^{e}
}
}
\label{d3} \tag{$3$}
\end{equation}
commute.
\end{theor}

\begin{remar}\label{r:exact}
By Korollar~3.4.6 of \cite{bonkat}, for every short exact sequence of extensions of separable $C^{*}$-algebras $e_{0} \overset{\phi}{\hookrightarrow} e_{1} \overset{\psi}{\twoheadrightarrow} e_{2}$ with completely positive coherent splitting, there exists a six term exact sequence
\begin{align*}
\xymatrix{
\kkE( e, e_{0} ) \ar[r]^{\phi_{*}} & \kkE ( e, e_{1} ) \ar[r]^{ \psi*} & \kkE ( e , e_{2} ) \ar[d]^{ \delta } \\
\kkE( e, e_{2} ) \ar[u]^{ \delta } & \kkE( e, e_{1} ) \ar[l]^{ \psi_{*}} & \kkE ( e, e_{0} ) \ar[l]^{\phi_{*}}
}
\end{align*}
of abelian groups.  Therefore, with $e = \mathfrak{e}_{n,i}$, the complex 
\begin{equation*}
\xymatrix{
H_{ n,i}^{ e_{0} } \ar[rr]^-{  \kE ( \phi )_{ \e{ n }{  i } } } & & H_{n, i}^{e_{1}} \ar[rr]^{ \kE ( \psi )_{ \e{ n  }{ i } } } & & H_{n,1}^{e_{2}}
}
\end{equation*}
is exact, for each $i=0,1,2,3,4,5$.
\end{remar}

\begin{lemma}\label{l:freekthy}
Let $e_{1}$ and $e_{2}$ be extensions of separable, nuclear $C^{*}$-algebras in the bootstrap category $\mathcal{N}$. 
Assume that the $K$-theory of $e_{1}$ is free, \ie, $F^{e_1}_{1,i}$ is free for all $i=0,1,2,3,4,5$.  Then the natural homomorphisms
\begin{equation*}
\Gamma_{e_{1}, e_{2} } \colon \kkE ( e_{1} , e_{2} ) \rightarrow  \Hom_{ \Lambda } ( \underline{K}_{\mathcal{E} } ( e_{1} ), \underline{K}_{\mathcal{E}} ( e_{2} ) )
\end{equation*}
and 
\begin{equation*}
\Delta_{ e_{1}, e_{2} } \colon  \Hom_{ \Lambda } ( \underline{K}_{\mathcal{E} } ( e_{1} ), \underline{K}_{\mathcal{E}} ( e_{2} ) ) \rightarrow  \Hom_{ \mathrm{six} } ( K_{\mathrm{six} } ( e_{1} ), K_{\mathrm{six}} ( e_{2} ) )
\end{equation*}
are isomorphisms.
\end{lemma}
 
\begin{proof}
By the UCT of Bonkat \cite{bonkat}, the homomorphism $\kkE ( e_{1} , e_{2} ) \rightarrow \Hom_{\mathrm{six} } ( \ksix ( e_{1} ), \ksix ( e_{2} ) )$ is an isomorphism.  Hence, the first map is injective and the second is surjective.  Therefore, it is enough to prove that the map $\Hom_{ \Lambda } ( \underline{K}_{\mathcal{E} } ( e_{1} ), \underline{K}_{\mathcal{E} } ( e_{2} ) ) \rightarrow \Hom_{\mathrm{six} } ( \ksix ( e_{1} ), \ksix ( e_{2} ) )$ is injective.  

So assume that $\alpha \in \Hom_{ \Lambda } ( \underline{K}_{\mathcal{E} } ( e_{1} ), \underline{K}_{\mathcal{E} } ( e_{2} ) )$ and $\alpha$ is zero on $\ksix ( e_{1} )$.  Let $\mathfrak{A}_{0} \hookrightarrow \mathfrak{A}_{1} \twoheadrightarrow \mathfrak{A}_{2}$ denote the extension $e_{1}$.  Since $F_{1,i+3}^{e_1}$ is free, by the last six term exact sequence of Theorem \ref{err:idealkthy}, $\rho_{n,i}^{e_{1}}$ is surjective.  Hence, $\alpha_{\F{n}{i} } = 0$.   

Let $x \in H_{n, i}^{e_{1}}$.  Since $\rho_{n,i+1}^{ e_{1}}$ is surjective, there exists $w \in F_{1, i+1}^{e_{1}}$ such that $\rho_{n,i+1}^{e_{1} } (w) =  h_{n, i, e_{1}}^{1, n, out } (x)$.  By Diagram (\ref{d2}) in Theorem \ref{err:idealkthy}, 
\begin{equation*}
h_{n,i, e_{1} }^{ 1, n, out } ( h_{n,i, e_{1}}^{1,1, in} (w) ) = \rho_{n,i+1}^{ e_{1} } (w). 
\end{equation*}
Therefore, $x - h_{n,i,e_{1}}^{1,1, in } ( w ) \in \ker( h_{n,i,e_{1}}^{1,n,out} ) = \im ( h_{n,i,e_{1}}^{1,n,in} )$.  Thus, there exists $y \in F_{1, i+2}^{e_{1}}$ such that 
\begin{equation*}
x = h_{n,i,e_{1}}^{1,1, in } ( w ) + h_{n,i, e_{1} }^{1,n,in} (y).
\end{equation*} 
Hence, $\alpha_{ \e{n}{i}  } ( x ) = 0$.
\end{proof}

\begin{lemma}\label{l:suspension}
Let $e_{1} \colon \ext{\mathfrak{A}_{0}}{\mathfrak{A}_{1}}{\mathfrak{A}_{2}}$ and $e_{2} \colon \ext{\mathfrak{B}_{0}}{\mathfrak{B}_{1}}{\mathfrak{B}_{2}}$ be extensions of separable, nuclear $C^{*}$-algebras in the bootstrap category $\mathcal{N}$. Then $\Gamma_{e_{1}, e_{2}}$ is an isomorphism if and only if $\Gamma_{\mathsf{S}e_{1}, \mathsf{S}e_{2}}$ is an isomorphism.
\end{lemma}

\begin{proof}
Note that the homomorphism $\mathsf{S}$ from $\Homsix(\ksix(e_1),\ksix(e_2))$ to $\Homsix(\ksix(\mathsf{S}e_1),\ksix(\mathsf{S}e_2))$ is an isomorphism (since $\mathsf{S}^2$ from $\Homsix(\ksix(e_1),\ksix(e_2))$ to $\Homsix(\ksix(\mathsf{S}^2e_1),\ksix(\mathsf{S}^2e_2))$ is an isomorphism). The same is true for the derived functor. By Bonkat's UCT  and the five lemma, it follows that the homomorphism $\hat{\mathsf{S}}$ from $\kkE(e_1,e_2)$ to $\kkE(\mathsf{S}e_1,\mathsf{S}e_2)$ constructed in \cite{err:idealkthy} is an isomorphism. It also follows from the five lemma, that the homomorphism $\mathsf{S}$ from $\Hom_\Lambda(\underline{K}_{\mathcal{E}}(e_1),\underline{K}_{\mathcal{E}}(e_2))$ to $\Hom_\Lambda(\underline{K}_{\mathcal{E}}(\mathsf{S}e_1),\underline{K}_{\mathcal{E}}(\mathsf{S}e_2))$ is an isomorphism. 
\end{proof}

\section{Geometric projective resolutions for extensions of Kirchberg algebras}

In this section, we will -- for each  given cyclic six term exact sequence $(G_{i})_{i=0}^5$ of finitely generated abelian groups with $G_3=0$ and $G_5$ being a free abelian group -- construct a very specific, concrete projective resolution and realize it as coming from a short exact sequence of extensions of simple \cstar-algebras. 

First we note that the homomorphism that sends $f$ to $\diag( \overbrace{f, f, \dots, f}^{m} )$ induces a morphism from $\mathfrak{e}_{n,0}$ to $\mathfrak{e}_{nm, 0}$ and a morphism from $\mathfrak{f}_{n,0}$ to $\mathfrak{f}_{nm, 0}$.  Therefore, we have homomorphisms 
\begin{align*}
\ftn{\kappa_{n, mn, i }^{e} }{ F_{mn, i}^{e} }{ F_{n, i}^{e} } \\
\ftn{\omega_{n, mn, i}^{e} }{ H_{mn, i}^{e} }{ H_{n, i}^{e} }.  
\end{align*}

By identifying $\mathbb{I}_{nm,0}$ with the sub-$C^{*}$-algebra 
\begin{equation*}
\setof{ f \in C_{0} ( (0,1] , \mathsf{M}_{n} \otimes \mathsf{M}_{m} ) }{ f(1) \in \C 1_{ \mathsf{M}_{n } } \otimes \C 1_{ \mathsf{M}_{m} } }
\end{equation*}
we get morphisms from $\mathfrak{e}_{mn, 0}$ to $\mathsf{M}_{n} ( \mathfrak{e}_{m, 0 } )$ and from $\mathfrak{f}_{mn, 0}$ to $\mathsf{M}_{n} ( \mathfrak{f}_{m,0} )$.  Hence, we get homomorphisms
\begin{align*}
\ftn{\varkappa_{mn, m, i }^{e} }{ F_{m, i}^{e} }{ F_{nm, i}^{e} } \\
\ftn{\chi_{mn, m, i}^{e} }{ H_{m, i}^{e} }{ H_{nm, i}^{e} } . 
\end{align*}

\begin{lemma}\label{l:bmaps}
Let $e \colon \mathfrak{A}_{0} \hookrightarrow \mathfrak{A}_{1} \twoheadrightarrow \mathfrak{A}_{2}$ be an extension of separable, nuclear \cstar-algebras.  Then
\begin{align*}
&\kappa_{n,mn, i }^{e} \circ \rho_{nm, i}^{ e} = \rho_{n, i }^{ e} &\varkappa_{mn, m, i}^{e} \circ \rho_{m, i }^{ e } = n \rho_{mn, i }^{ e }  \\
&\beta_{n, i }^{ e } \circ \kappa_{n, mn, i}^{e} = m \beta_{nm, i }^{ e}  &\beta_{nm, i }^{ e } \circ \varkappa_{nm, m, i}^{e} = \beta_{m, i }^{ e}
\end{align*}
and the following diagrams 
$$\xymatrix{
    F_{1,i+2}^{e} \ar[r]^{h_{mn, i, e}^{1, mn, in } } \ar[d] & H_{mn, i}^{e} \ar[r]^{h_{mn, i, e}^{1, nm, out } } \ar[d]^{ \omega^{e}_{n, mn, i} } & F_{mn, i+1} \ar[d]^{ \kappa^{e}_{n, mn, i+1} } 
    & 
    F_{1,i+2}^{e} \ar[r]^{h_{m, i, e}^{1, m, in } } \ar[d] & H_{m,i}^{e} \ar[r]^{h_{m, i, e}^{1, m, out } } \ar[d]^{ \chi^{e}_{mn, m, i} }& F_{m, i+1} \ar[d]^{ \varkappa^{e}_{mn, m,i+1} } \\
    F_{1,i+2}^{e} \ar[r]_{h_{n, i, e}^{1, n, in } } & H_{n,i}^{e} \ar[r]_{h_{n, i, e}^{1, n, out } }  & F_{n, i+1}
    &
    F_{1,i+2}^{e} \ar[r]_{h_{mn, i, e}^{1, mn, in } } & H_{mn, i}^{e} \ar[r]_{h_{mn, i, e}^{1, nm, out } } & F_{mn, i+1} \\
    F_{1,i+1}^{e} \ar[r]^{h_{mn, i, e}^{1, 1, in} } \ar@{=}[d] & H_{mn, i}^{e}  \ar[r]^{h_{mn, i, e}^{1, 1 , out } } \ar[d]^{ \omega^{e}_{ n, mn, i} } & F_{1, i+3}  \ar[d]^{\times m} 
    &
    F_{1,i+1}^{e} \ar[r]^{h_{m, i, e}^{1, 1, in} } \ar[d]^{\times n} & H_{m, i}^{e}  \ar[r]^{h_{m, i, e}^{1, 1 , out } } \ar[d]^{ \chi_{ mn, m, i}^{e} } & F_{1, i+3}  \ar@{=}[d] \\
    F_{1,i+1}^{e} \ar[r]_{h_{n, i, e}^{1, 1, in} } &  H_{n, i}^{e}  \ar[r]_{h_{n, i, e}^{1, 1 , out } } & F_{1, i+3}  
    &
    F_{1,i+1}^{e} \ar[r]_{h_{mn, i, e}^{1, 1, in} } &  H_{mn, i}^{e}  \ar[r]_{h_{mn, i, e}^{1, 1 , out } } & F_{1, i+3}  
}$$
are commutative for $i=0,1,2,3,4,5$. 
\end{lemma}

\begin{proof}
The first four equations follow just like in ordinary $K$-theory with coefficients. 
For every extension $e' \colon \mathfrak{A}_{0}' \hookrightarrow \mathfrak{A}_{1}' \twoheadrightarrow \mathfrak{A}_{2}'$ of \cstar-algebras, we let $\mathfrak{i}(e')$ and $\mathfrak{q}(e')$ denote the extensions $\mathfrak{A}_{0}'\overset{\operatorname{id}}{\hookrightarrow}\mathfrak{A}_{0}'\twoheadrightarrow 0$ and $0\hookrightarrow\mathfrak{A}_{2}'\overset{\operatorname{id}}{\twoheadrightarrow} \mathfrak{A}_{2}'$, respectively (\cite[Section~2]{err:idealkthy}).  
For commutativity of the four diagrams, we note the following. 
\begin{enumerate}[(1)]
\item
  Note that the map $f \mapsto \operatorname{diag}( f, f, \dots, f )$ induces a homomorphism 
  \begin{align*}
    \xymatrix{
      \mathfrak{i} ( \mathfrak{e}_{n,2} ) \ar@{^{(}->}[r] \ar[d] & \mathfrak{e}_{n,2} \ar@{>>}[r] \ar[d] & \mathfrak{q} ( \mathfrak{e}_{n,2} ) \ar[d] \\
      \mathfrak{i} ( \mathfrak{e}_{nm,2} ) \ar@{^{(}->}[r] & \mathfrak{e}_{nm,2} \ar@{>>}[r]  & \mathfrak{q} ( \mathfrak{e}_{nm,2} ) 
    }
  \end{align*}
of extensions.  Applying $\mathfrak{mc}^{i}$ to this diagram, for $i=0,1,2,3,4,5$, we get the first diagram.
\item
  Note that the homomorphism from $\mathfrak{e}_{nm,0}$ to $\mathsf{M}_{n} ( \mathfrak{e}_{m,0} )$ induces a homomorphism 
  \begin{align*}
    \xymatrix{
      \mathfrak{i} ( \mathfrak{e}_{nm,2} ) \ar@{^{(}->}[r] \ar[d] & \mathfrak{e}_{nm,2} \ar@{>>}[r] \ar[d]  & \mathfrak{q} ( \mathfrak{e}_{nm,2} ) \ar[d] \\
      \mathfrak{i} \left( \mathsf{M}_{n} ( \mathfrak{e}_{m,2} ) \right) \ar@{^{(}->}[r]  & \mathsf{M}_{n} ( \mathfrak{e}_{m,2} ) \ar@{>>}[r] & \mathfrak{q} \left( \mathsf{M}_{n} ( \mathfrak{e}_{m,2} ) \right) 
    }
  \end{align*}
of extensions.  Applying $\mathfrak{mc}^{i}$ to this diagram, for $i=0,1,2,3,4,5$, we get the second diagram. 
\item
  Note that the map $f \mapsto \operatorname{diag}( f, f, \dots, f )$ induces a homomorphism 
  \begin{align*}
    \xymatrix{
      \mathfrak{i} ( \mathfrak{e}_{n,0} ) \ar@{^{(}->}[r] \ar[d] & \mathfrak{e}_{n,0} \ar@{>>}[r] \ar[d] & \mathfrak{q} ( \mathfrak{e}_{n,0} ) \ar@{=}[d] \\
      \mathfrak{i} ( \mathfrak{e}_{nm,0} ) \ar@{^{(}->}[r] & \mathfrak{e}_{nm,0} \ar@{>>}[r]  & \mathfrak{q} ( \mathfrak{e}_{nm,0} ) 
    }\end{align*}of extensions.  Applying $\mathfrak{mc}^{i}$ to this diagram, for $i=0,1,2,3,4,5$, we get the third diagram. 
\item
  Note that the homomorphism from $\mathfrak{e}_{nm,0}$ to $\mathsf{M}_{n} ( \mathfrak{e}_{m,0} )$ induces a homomorphism 
  \begin{align*}
    \xymatrix{
      \mathfrak{i} ( \mathfrak{e}_{nm,0} ) \ar@{^{(}->}[r] \ar@{=}[d] & \mathfrak{e}_{nm,0} \ar@{>>}[r] \ar[d]  & \mathfrak{q} ( \mathfrak{e}_{nm,0} ) \ar[d] \\
      \mathfrak{i} \left( \mathsf{M}_{n} ( \mathfrak{e}_{m,0} ) \right) \ar@{^{(}->}[r]  & \mathsf{M}_{n} ( \mathfrak{e}_{m,0} ) \ar@{>>}[r] & \mathfrak{q} \left( \mathsf{M}_{n} ( \mathfrak{e}_{m,0} ) \right) 
    }
  \end{align*}
of extensions.  Applying $\mathfrak{mc}^{i}$ to this diagram, for $i=0,1,2,3,4,5$, we get the fourth diagram. \qedhere
\end{enumerate}
\end{proof}

\begin{defin}
  By a \bolddefine{Kirchberg algebra} we mean a simple, purely infinite,
  separable, nuclear $C^{*}$-algebra.
\end{defin}

\begin{remar}
  \label{rem:strong-pure-infiniteness}
  Strong pure infiniteness is
  considered in \cite{kirchrordam2}, 
  and it is shown that a separable, stable, nuclear $C^{*}$-algebra \A is
  strongly purely infinite if and only if \A absorbs $\mathcal{O}_\infty$, \ie, if and only if $\A\cong\A\otimes\mathcal{O}_\infty$. 
  
  It is shown independently by Kirchberg, 
  and by Toms and Winter (\cf\ \cite[Theorem~4.3]{tomswinterZstable}) 
  that $\mathcal{O}_\infty$-stability passes to extensions, \ie, 
  if $\A$ and $\B$ are $\mathcal{O}_\infty$-stable, so is every extension
  of \A and \B. The opposite was shown by Kirchberg and R\o rdam 
  (\cf\ \cite[Proposition~8.5]{kirchrordam2}). 
  
  Thus we can note that a separable, stable, nuclear $C^{*}$-algebras with finitely many
  ideals is strongly purely infinite if and only if it is
  $\mathcal{O}_\infty$-absorbing if and only if it is in the smallest class
  of $C^{*}$-algebras closed under extensions and containing all stable Kirchberg
  algebras. 

  In particular, every extension of a Kirchberg algebra by another Kirchberg algebra is strongly purely
  infinite. 
\end{remar}

\begin{propo}\label{p:projres}
Let $( G_{i} )_{ i = 0 }^{5}$ be a cyclic six term exact sequence of countable abelian groups, and let 
\begin{equation}\label{eqproj}
( H_{i} )_{ i=  0 }^{5} \hookrightarrow ( F_{i} )_{ i =0}^{5} \twoheadrightarrow ( G_{i} )_{ i = 0}^{5}
\end{equation}
be any projective resolution with countable abelian groups, only.  Then there exists a short exact sequence 
\begin{equation}\label{eqgeoproj}
\mathsf{S} e  \hookrightarrow  e'' \twoheadrightarrow  e'
\end{equation}
of extensions where $e$ and $e'$ are extensions of stable Kirchberg algebras in the bootstrap category $\mathcal{N}$, and $e''$ is an extension of nuclear, separable $C^{*}$-algebras in the bootstrap category $\mathcal{N}$, such that the cyclic six term exact sequence associated with (\ref{eqgeoproj}) degenerates into short exact sequences 
\begin{equation*}
\ksix ( e'' ) \hookrightarrow \ksix ( e' ) \twoheadrightarrow \ksix ( \mathsf{S}^{2} e )
\end{equation*}
which are isomorphic to the projection resolution in (\ref{eqproj}) (in such a way that $F_{1,i}^{e''}\cong H_i$, $F_{1,i}^{e'}\cong F_i$, and $F_{1,i}^{\mathsf{S}^2 e}\cong G_i$ for $i=0,1,2,3,4,5$, of course).
\end{propo}

\begin{proof}
According to Proposition 5.4 of \cite{extpurelyinf}, there exist essential extensions $e$ and $e'$ of stable Kirchberg algebras in the bootstrap category $\mathcal{N}$, such that $\ksix ( e )$ and $\ksix ( e' )$ are (isomorphic) to the sequences $( G_{i} )_{ i = 0 }^{5}$ and $( F_{i} )_{ i = 0 }^{5}$, resp.  

We have an epic morphism $\ftn{ \phi }{ \ksix ( e' ) }{ \ksix ( e ) }$ given from (\ref{eqproj}).  Using the UCT of Bonkat (see Satz 7.5.3 of \cite{bonkat}), we can lift this to an element $x \in \kkE ( e, e ' )$.  Using Kirchberg's lifting result for nuclear, stable, strongly purely infinite $C^{*}$-algebras (see Hauptsatz 4.2 of \cite{kirchpure}), we know that $x$ can be lifted to a morphism $\ftn{ \psi = ( \psi_{0} , \psi_{1} , \psi_{2} ) }{ e }{ e' }$ of extensions.  

Consequently, $\ksix ( \psi ) = \phi$.  Let now
\begin{equation*}
\mathsf{S} e \hookrightarrow e'' \twoheadrightarrow e'
\end{equation*} 
be the mapping cone sequence of the morphism $\psi$ (see \cite{bonkat}).  Then $e''$ is an extension of nuclear, separable $C^{*}$-algebras in the bootstrap category $\mathcal{N}$.  Since the index and exponential maps correspond to the epic morphism $\ksix ( \psi )$, the cyclic six term exact sequence degenerates into short sequences of the form 
\begin{equation*}
\ksix ( e'' ) \hookrightarrow \ksix ( e' ) \twoheadrightarrow \ksix ( \mathsf{S}^{2} e ).
\end{equation*}
Necessarily, $\ksix ( e'' )$ has to correspond to $( H_{i} )_{ i = 0}^{5}$ since this is (up to isomorphism) the kernel.
\end{proof}

\begin{assumption}\label{assume1}
For the rest of this section, let 
\begin{equation*}
\xymatrix{
G_{0} \ar[r] & G_{1} \ar[r] & G_{2} \ar[d] \\
G_{5} \ar[u] & G_{4} \ar[l] & G_{3} \ar[l]
}
\end{equation*}
be a given cyclic six term exact sequence of finitely generated abelian groups.  Assume moreover, that $G_{3} = 0$ and that $G_{5}$ is a free abelian group.  Consequently, $G_{4}$ is also free.
\end{assumption}

According to the structure theorem for finitely generated abelian groups, we can -- up to isomorphism -- write $G_{2}$ as 
\begin{align*}
&( \Z_{ p_{1}^{ n_{1 , 1} } } )^{ m_{1, 1} } \oplus ( \Z_{ p_{1}^{ n_{1 , 2}} } )^{ m_{1,2} } \oplus \cdots \oplus ( \Z_{ p_{1}^{ n_{1} , k_{1} } } )^{ m_{ 1, k_{1} } } \\
&\oplus ( \Z_{ p_{2}^{ n_{2,1} } } )^{ m_{2,1} } \oplus ( \Z_{ p_{2}^{ n_{2,2} } } )^{ m_{2,2} } \oplus \cdots \oplus ( \Z_{ p_{2}^{ n_{2, k_{2} } } } )^{ m_{2, k_{2} } } \\
&\oplus \cdots \oplus ( \Z_{ p_{l}^{ n_{\ell,1} } } )^{ m_{\ell,1} } \oplus ( \Z_{ p_{\ell}^{ n_{\ell,2} } } )^{ m_{\ell,2} } \oplus \cdots \oplus (\Z_{ p_{\ell}^{ n_{ \ell , k_{\ell} } } } )^{ m_{\ell, k_{\ell} } } \\
&\oplus \Z^{s}
\end{align*}
where $\ell,s \in \N_{0}$ and $(p_{i} )_{ i = 1}^{\ell}$ is a strictly increasing sequence of prime numbers, $( k_{j} )_{ j = 1}^{\ell}$, $(m_{i,j})_{ j = 1}^{ k_{i} } \subset \N$, for all $i = 1, \dots, \ell$, and $( n_{i,j} )_{ j = 1}^{ k_{i} } \subset \N$ is strictly increasing, for all $i = 1, \dots , \ell$.

For each canonical generator of this group, we fix a lifting to an element of $G_{1}$.  Let $F_{2}$ denote the group
\begin{align*}
&\Z^{ m_{1, 1} } \oplus \Z^{ m_{1,2} } \oplus \cdots \oplus \Z^{ m_{ 1, k_{1} } } \\
&\oplus \Z^{ m_{2,1} } \oplus \Z^{ m_{2,2} } \oplus \cdots \oplus \Z^{ m_{2, k_{2} } } \\
&\oplus \cdots \oplus \Z^{ m_{\ell,1} } \oplus \Z^{ m_{\ell,2} } \oplus \cdots \oplus \Z^{ m_{\ell, k_{\ell} } } \\
&\oplus \Z^{s}.
\end{align*}
Let $\ftn{ \zeta }{ F_{2} }{ G_{1} }$ be the surjective homomorphism from $F_{2}$ to $G_{1}$ which sends each canonical generator to the lifting of the corresponding generator of $G_{2}$ (chosen above).

There exists a finitely generated free abelian group $\widetilde{F}_{0}$ and a surjective homomorphism $\ftn{\eta_{0}}{ \widetilde{F}_{0} }{ G_{0} }$.  Then we have a projective resolution of $( G_{i} )_{ i = 0 }^{5}$ as follows:
\begin{equation*}
\xymatrix{
\ar[r] & 0 \ar[r] \ar[d]^{ \lambda_{3} } & 0 \ar[r] \ar[d]^{ \lambda_{4} } & H_{5} \ar[r] \ar@{^{(}->}[d]^{ \lambda_{5}} & H_{0} \ar[r] \ar@{^{(}->}[d]^{ \lambda_{0} } & H_{1} \ar[r] \ar@{^{(}->}[d]^{ \lambda_{1} }& H_{2} \ar[r] \ar@{^{(}->}[d]^{ \lambda_{2} } & \\
\ar[r] & 0 \ar[r] \ar[d]^{ \eta_{3} } & G_{4} \ar[r] \ar@{=}[d]^{ \eta_{4} } & G_{4} \oplus G_{5} \ar[r] \ar@{>>}[d]^{ \eta_{5} } & G_{5} \oplus \widetilde{F}_{0} \ar[r] \ar@{>>}[d]^{ \eta_{0} }& \widetilde{F}_{0} \oplus F_{2} \ar[r] \ar@{>>}[d]^{ \eta_{1} } & F_{2} \ar[r]  \ar@{>>}[d]^{ \eta_{2} }& \\
\ar[r] & 0 \ar[r] & G_{4} \ar[r] & G_{5} \ar[r] & G_{0} \ar[r] & G_{1} \ar[r] & G_{2} \ar[r] & 
}
\end{equation*}
where $( \eta_{i} )_{ i= 0}^{5}$ are defined in the obvious way (we use, of course, $\zeta$ in the definition of $\eta_{1}$), and the $( H_{i} )_{ i = 0 }^{5}$ is the kernel of $( \eta_{i} )_{ i= 0 }^{5}$.  An easy diagram chase shows that $\eta_{1}$ is surjective.  To match the notation from Proposition \ref{p:projres}, we let $F_{3} = 0$, $F_{4} = G_{4}$, $F_{5} = G_{4} \oplus G_{5}$, $F_{0} = G_{5} \oplus \widetilde{F}_{0}$, and $F_{1} = \widetilde{F}_{0} \oplus F_{2}$.

We will, in the canonical way, identify $H_{2}$ with 
\begin{align*}
&\Z^{ m_{1, 1} } \oplus \Z^{ m_{1,2} } \oplus \cdots \oplus \Z^{ m_{ 1, k_{1} } } \\
&\oplus \Z^{ m_{2,1} } \oplus \Z^{ m_{2,2} } \oplus \cdots \oplus \Z^{ m_{2, k_{2} } } \\
&\oplus \cdots \oplus \Z^{ m_{\ell,1} } \oplus \Z^{ m_{\ell,2} } \oplus \cdots \oplus \Z^{ m_{\ell, k_{\ell} } } \\
\end{align*} 
and 
under this identification the injection $\ftn{ \lambda_{2} }{ H_{2} }{ F_{2} }$ can be identified with the diagonal matrix
\begin{align*}
\mathrm{diag}(& \overbrace{p_{1}^{n_{1,1} } , \cdots, p_{1}^{n_{1,1} } }^{ m_{1,1} }, \overbrace{p_{1}^{ n_{1,2} } , \cdots p_{1}^{n_{1,2}} }^{ m_{1,2} }, \cdots , \overbrace{ p_{1}^{n_{1,k_{1}}} , \cdots , p_{1}^{ n_{1, k_{1} } } }^{ m_{1, k_{1} } } , \\
		 & \overbrace{ p_{2}^{n_{2,1} } , \cdots, p_{2}^{n_{2,1} } }^{ m_{2,1} } ,\overbrace{ p_{2}^{ n_{2,2} } , \cdots p_{2}^{n_{2,2}} }^{ m_{2,2} }, \cdots , \overbrace{ p_{2}^{n_{2,k_{2}}} , \cdots , p_{2}^{ n_{2, k_{2} } } }^{ m_{ 2 ,k_{2} } }, \\
		 & \cdots , \overbrace{ p_{\ell}^{n_{\ell,1} } , \cdots, p_{\ell}^{n_{\ell,1} } }^{ m_{\ell,1} }, \overbrace{ p_{\ell}^{ n_{\ell,2} } , \cdots p_{\ell}^{n_{\ell,2}} }^{ m_{ \ell,2 } }, \cdots , \overbrace{ p_{\ell}^{n_{\ell,k_{\ell}}} , \cdots , p_{\ell}^{ n_{\ell, k_{\ell} } } }^{ m_{\ell, k_{\ell} } } ).
\end{align*}

\section{Exactness of a sequence}
If we let $e_1$ and $e_2$ be essential extensions of stable Kirchberg algebras in the bootstrap category $\mathcal{N}$ with all the $K$-theory appearing in the cyclic six term exact sequences being finitely generated and the $K_1$-groups of the quotients being free and the $K_1$-groups of the ideals being zero. 
Now let $\mathsf{S} e_{1} \overset{\phi}{\hookrightarrow} e''  \overset{\psi}{ \twoheadrightarrow}  e'$ be a short exact sequence of essential extensions of stable Kirchberg algebras in the bootstrap category $\mathcal{N}$ such that the induced sequence $\ksix ( e'' ) \hookrightarrow \ksix ( e' ) \twoheadrightarrow \ksix ( \mathsf{S}^{2} e_{1} )$ is exactly (isomorphic to) the projective resolution given in the previous section. 
Then this induces a cyclic six term sequence
$$\xymatrix{
\underline{K}_{\mathcal{E}}(\mathsf{S}e_1)\ar[r] & \underline{K}_{\mathcal{E}}(e'')\ar[r] & \underline{K}_{\mathcal{E}}(e')\ar[d] \\ 
\underline{K}_{\mathcal{E}}(\mathsf{S}e')\ar[u] & \underline{K}_{\mathcal{E}}(\mathsf{S}e'')\ar[l]  & \underline{K}_{\mathcal{E}}(\mathsf{S}^2e_1)\ar[l] 
}$$
In this chapter we show, that when we apply the functor $\Hom_{ \Lambda} ( \underline{K}_{ \mathcal{E} } (-) , \underline{K}_{ \mathcal{E} } ( \mathsf{S} e_{2} ) )$ to this sequence, the sequence 
\begin{equation*}
\xymatrix{ \Hom_{ \Lambda} ( \underline{K}_{ \mathcal{E} } ( e ' ) , \underline{K}_{ \mathcal{E} } ( \mathsf{S} e_{2} ) ) \ar[r]^{\psi^{*} } 
& \Hom_{ \Lambda } ( \underline{K}_{\mathcal{E} } ( e'' ) , \underline{K}_{\mathcal{E} } ( \mathsf{S} e_{2} ) ) \ar[r]^{\phi^{*} } 
& \Hom_{ \Lambda } ( \underline{K}_{\mathcal{E} } ( \mathsf{S}e_{1} ) , \underline{K}_{ \mathcal{E} } (\mathsf{S}e_{2} ) ) \ar[d]^{ \delta } \\
&&\Hom_{ \Lambda } ( \underline{K}_{\mathcal{E} } ( \mathsf{S}e' ) , \underline{K}_{ \mathcal{E} } ( \mathsf{S}e_{2} ) ) }
\end{equation*}
is exact. This is the main ingredient in the proof of the UMCT in the next section. 

\begin{assumption}\label{assume2}
Let $( G_{i} )_{ i = 0}^{5}$ be as in Assumption \ref{assume1}, and let 
\begin{equation*}
\xymatrix{( H_{i} )_{ i = 1}^{5} \ar@{^{(}->}[r]^-{ (\lambda_{i})_{ i = 0 }^{5} }  & ( F_{i} )_{ i = 0 }^{5} \ar@{>>}[r]^-{( \eta_{i} )_{ i = 0 }^{5} }  & ( G_{i} )_{ i = 0 }^{5} }
\end{equation*} 
be the projective resolution given right after Assumption \ref{assume1}.

Let, moreover, $\mathsf{S} e_{1} \overset{\phi}{\hookrightarrow} e''  \overset{\psi}{ \twoheadrightarrow}  e'$ be a geometric resolution corresponding exactly to this resolution according to Proposition \ref{p:projres}, \ie, the induced sequence $\ksix ( e'' ) \hookrightarrow \ksix ( e' ) \twoheadrightarrow \ksix ( \mathsf{S}^{2} e_{1} )$ is exactly (isomorphic to) the above resolution.

Let 
\begin{equation*}
\mathfrak{A}_{0} \overset{\iota}{\hookrightarrow} \mathfrak{A}_{1} \overset{ \pi }{ \twoheadrightarrow } \mathfrak{A}_{2}, \quad \mathfrak{A}_{0}' \overset{\iota'}{\hookrightarrow} \mathfrak{A}_{1}' \overset{ \pi' }{ \twoheadrightarrow } \mathfrak{A}_{2}', \quad \text{and} \quad \mathfrak{A}_{0}'' \overset{\iota''}{\hookrightarrow} \mathfrak{A}_{1}'' \overset{ \pi'' }{ \twoheadrightarrow } \mathfrak{A}_{2}''
\end{equation*}
denote the extensions $e_{1}$, $e'$, and $e''$, resp.
\end{assumption}

\begin{lemma}\label{l:kken1}
Suppose $n = p_{i}^{n_{i,j}}$.  Then $H_{n, 1}^{e''}$ is isomorphic to 
\begin{align*}
&(\Z_{n})^{ m_{1, 1} } \oplus (\Z_{n })^{ m_{1,2} } \oplus \cdots \oplus (\Z_{n })^{ m_{ 1, k_{1} } } \\
&\oplus (\Z_{n })^{ m_{2,1} } \oplus (\Z_{n })^{ m_{2,2} } \oplus \cdots \oplus (\Z_{n})^{ m_{2, k_{2} } } \\
&\oplus \cdots\oplus (\Z_{n })^{ m_{\ell,1} } \oplus (\Z_{n })^{ m_{\ell,2} } \oplus \cdots \oplus (\Z_{n })^{ m_{\ell, k_{\ell} } },
\end{align*}
$H_{n,1}^{e'}$ is isomorphic to 
\begin{align*}
&(\Z_{n})^{ m_{1, 1} } \oplus (\Z_{n })^{ m_{1,2} } \oplus \cdots \oplus (\Z_{n })^{ m_{ 1, k_{1} } } \\
&\oplus (\Z_{n })^{ m_{2,1} } \oplus (\Z_{n })^{ m_{2,2} } \oplus \cdots \oplus (\Z_{n })^{ m_{2, k_{2} } } \\
&\oplus \cdots \oplus (\Z_{n })^{ m_{\ell,1} } \oplus (\Z_{n })^{ m_{\ell,2} } \oplus \cdots \oplus (\Z_{n })^{ m_{\ell, k_{\ell} } } \\
&\oplus (\Z_{n})^{s},
\end{align*}
and we may identify $\ker ( \ftn{ \underline{K}_{\mathcal{E}}(\psi)_{\mathfrak{e}_{n,1}} } { H_{n,1}^{e''} }{ H_{n,1}^{e'} } )$ with 
\begin{align*}
    &0\oplus 0\oplus\cdots\oplus 0\\
    &\oplus\cdots\\
    &\oplus p_i^{n_{i,j}-n_{i,1}}(\Z_n)^{m_{i,1}}
    \oplus \cdots
    \oplus p_i^{n_{i,j}-n_{i,j-1}}(\Z_n)^{m_{i,j-1}}
    \oplus (\Z_n)^{m_{i,j}}
    \oplus\cdots
    \oplus(\Z_n)^{m_{i,k_i}}\\
    &\oplus \cdots \\
    &\oplus 0\oplus 0\oplus\cdots\oplus 0.
  \end{align*}
\end{lemma}

\begin{proof}
  Note that $\ftn{ n f_{1, 4}^{ e'} }{ F_{1, 4 }^{e'} }{ F_{1,5}^{e'} }$ and $\ftn{ n f_{1, 4}^{ e''} }{ F_{1,4}^{e''} }{ F_{1,5}^{e''} }$ are injective homomorphisms.  Therefore, by Theorem \ref{err:idealkthy}, 
  \begin{align}\label{exact1}
    \xymatrix{
      F_{1,1}^{e'} \ar[r]^-{ n f_{1, 1}^{ e'} } & F_{1,2}^{ e'} \ar[r]^-{ h_{n, 1, e' }^{1,1, in } } & H_{n,1}^{e'} \ar[r] & 0 }
  \end{align}
  \begin{align}\label{exact2}
    \xymatrix{
      F_{1,1}^{e''} \ar[r]^-{ n f_{1, 1}^{ e'' } } & F_{1,2}^{ e''} \ar[r]^-{ h_{n, 1, e'' }^{1,1, in } } & H_{n,1}^{e''} \ar[r] & 0 
    }
  \end{align}
  are exact sequences.  Since $f_{1, 1}^{ e'}$ and $f_{1,1}^{e''}$ are surjective functions, $\im ( n f_{1,1}^{ e' } ) = n F_{1,2}^{e'}$ and $\im (n f_{1,1}^{ e'' } ) = n F_{1,2}^{e''}$.  Thus, $H_{n, 1}^{e''}$ and $H_{n, 1}^{e'}$ are as stated.
  
  Recall that $\ftn{ \kE ( \psi )_{ \F{1}{ 2}  } }{ F_{1,2}^{e''} }{ F_{1,2}^{e'} }$ is equal to $\lambda_{2}$.  Suppose 
  \begin{align*}
    x 
    &= (x_{1,1 } , x_{1,2} , \ldots, x_{ 1, k_{1} } , \\
    &\quad\quad x_{2,1} , x_{2,2} , \ldots, x_{2, k_{2} }, \\
    &\quad\quad \ldots, x_{\ell, 1 } , x_{\ell,2} , \ldots , x_{\ell , k_{\ell} } )
  \end{align*}
  is an element of $F_{1,2}^{e''}$ such that $\kE ( \psi )_{ \e{n}{1} }  ( h_{n,1, e'' }^{1,1, in } ( x ) ) = 0$.  Since $h_{n,1, e' }^{1,1,in } ( \lambda_{2}(x) ) = \kE ( \psi )_{ \e{n}{1} } ( h_{n,1,e''}^{1,1,in} ( x ) ) = 0$, by Equation (\ref{exact1}), $\lambda_{2} ( x ) = n f_{1,1}^{ e'} (y)$ for some $y \in F_{1,1}^{e'}$.  
  Since 
  \begin{align*}
    \lambda_{2} ( x ) 
    &= \big( p_1^{n_{1,1}} x_{1,1} , 
    p_1^{n_{1,2}} x_{1,2} , 
    \ldots,
    p_1^{n_{1,k_1}} x_{1,k_1} , \\
    &\quad\quad p_2^{n_{2,1}} x_{2,1} , 
    p_2^{n_{2,2}} x_{1,2} , 
    \ldots,
    p_2^{n_{2,k_2}} x_{2,k_2} , \\
    &\quad\quad\ldots, 
    p_\ell^{n_{\ell,1}} x_{\ell,1} , 
    p_\ell^{n_{\ell,2}} x_{\ell,2} , 
    \ldots,
    p_\ell^{n_{\ell,k_\ell}} x_{\ell,k_\ell} , \\
    &\quad\quad 0 \big) \\
    &= p_i^{n_{i,j}}
    \big( y_{1,1} , y_{1,2} , \ldots, y_{1,k_1} , \\
    &\qquad\qquad y_{2,1} , y_{2,2} , \ldots, y_{2,k_2}, \\
    &\qquad\qquad\ldots,
    y_{\ell,1} , y_{\ell,2}, \ldots, y_{\ell,k_\ell} , \\
    &\qquad\qquad 0 \big),
  \end{align*}
  we have that
  \begin{enumerate}
  \item 
    $x_{t,v} = p_i^{n_{i,j}}z_{t,v}$, 
    for $t \neq i$, 
  \item 
    $x_{i,v} = p_i^{n_{i,j} - n_{i,v}} y_{i,v}$,
    for $v < j$, 
  \item 
    $p_i^{n_{i,v}}x_{i,v}
    =p_i^{n_{i,j}}(p_i^{n_{i,v}-n_{i,j}}x_{i,v})$, 
    for $v \geq j$, 
    where $p_i^{n_{i,v}-{n_{i,j}}} x_{i,v}=y_{i,v}$ 
    is an element of $\Z^{m_{i,v}}$.  
  \end{enumerate}
  The lemma now follows.
\end{proof}

\begin{lemma}\label{l:red1}
Let $e_{2} \colon \mathfrak{B}_{0} \hookrightarrow \mathfrak{B}_{1} \twoheadrightarrow \mathfrak{B}_{2}$ be an extension of separable $C^{*}$-algebras.  If $\alpha$ is an element of $\Hom_{ \Lambda } ( \underline{K}_{\mathcal{E} } ( e_{1} ) , \underline{K}_{\mathcal{E}} ( e_{2} ) )$ and $\alpha_{ \e{p_{i}^{ n_{ i , j } }}{4} }  = 0$ for all $i= 1, \dots, \ell$ and $j = 1, \dots, k_{i}$ and $\alpha_{\F{1}{k} } = 0$ for all $k=0,1,\dots, 5$, then $\alpha_{ \e{n}{4} } = 0$ for all $n\in\N_{\geq 2}$.  
Consequently, if $\alpha \in \Hom_{ \Lambda } ( \underline{K}_{\mathcal{E} } ( \mathsf{S}e_{1} ) , \underline{K}_{\mathcal{E}} ( e_{2} ) )$ and $\alpha_{ \e{p_{i}^{ n_{ i , j } }}{1} } = 0$ for all $i= 1, \dots, \ell$ and $j = 1, \dots, k_{i}$ and $\alpha_{ \F{1}{k} } = 0$ for all $k=0,1,\dots, 5$, then $\alpha_{ \e{n}{1} }  = 0$ for all $n\in\N_{\geq 2}$.
\end{lemma}

\begin{proof}
\emph{Claim 1:} If $\alpha_{ \e{k}{4} }  = 0$ for all $k = p_{i}^{r}$, then $\alpha_{ \e {n}{4} }  = 0$ for all $n = p_{1}^{s_{1}} \cdots p_{\ell}^{s_{\ell}}$.  

\bigskip

Let $z \in H_{n,4}^{e_{1}}$ with $n = p_{1}^{s_{1}} \cdots p_{\ell}^{s_{\ell}}$.  Let $x = h_{n, 4, e_{1}}^{1, n, out } ( z ) \in F_{n, 5}^{e_{1}}$ and let $\overline{x} = \beta_{n, 5 }^{ e_{1} } ( x ) \in F_{1, 2}^{e_{1}}$.  Note that $\overline{x} = ( \overline{x}_{1}, \dots, \overline{x}_{\ell} )$, where $\overline{x}_{i}$ is in 
\begin{equation*}
0 \oplus \cdots \oplus 0 \oplus ( \Z_{ p_{i}^{n_{i, 1}} } )^{m_{i,1}} \oplus \cdots \oplus ( \Z_{ p_{i}^{n_{i, k_{i}}} } )^{m_{i,k_{i}}} \oplus 0 \oplus \cdots \oplus 0.
\end{equation*}
Since $p_{1}^{s_{1}} \cdots p_{\ell}^{s_{\ell}} \overline{x}_{i} = 0$ and $\gcd ( p_{i} , p_{t} ) = 1$ for all $i \neq t$, we have that $p_{i}^{s_{i}} \overline{x}_{i} = 0$.  By Theorem \ref{err:idealkthy}, there exists $y_{i} \in F_{p_{i}^{s_{i} },5}^{e_{1}}$ such that $\beta_{ p_{i}^{s_{i} } , 5 }^{ e_{1 } } ( y_{i} ) = \overline{x}_{i}$. 

By Lemma \ref{l:bmaps}
\begin{align*}
\beta_{n, 5 }^{ e_{1} } \left(  \sum_{ i = 1}^{\ell} \varkappa_{n, p_{i}^{s_{i} } , 5}^{e_{1} } ( y_{i} ) \right) &= \sum_{ i = 1}^{\ell} \beta_{ p_{i}^{s_{i} } , 5 }^{ e_{1 } } ( y_{i} ) 
	= \overline{x}= \beta_{n, 5 }^{ e_{1} } ( x ).
\end{align*}
By Theorem \ref{err:idealkthy}, there exists $a \in F_{1, 5}^{e_{1}}$ such that $x = \rho_{n, 5 }^{ e_{1} } (a ) + \sum_{ i = 1}^{\ell} \varkappa_{n, p_{i}^{s_{i} } , 5}^{e_{1} } ( y_{i} )$.  Since $F_{1,3}^{e_{1}} = 0$, by Theorem \ref{err:idealkthy}, we have that $h_{p_{i}^{s_{i}}, 4, e_{1} }^{1, p_{i}^{s_{i} } , out }$ is surjective.  Thus, there exists $b_{i} \in H_{ p_{i}^{s_{i} }, 4 }^{e_{1}}$ such that $h_{p_{i}^{s_{i}}, 4, e_{1} }^{1, p_{i}^{s_{i} } , out } (b_{i}) = y_{i}$.  By Lemma \ref{l:bmaps}, 
\begin{align*}
h_{n, 4, e_{1} }^{1, n, out} \left( \sum_{ i = 1}^{\ell} \chi_{n, p_{i}^{s_{i} } , 4 }^{e_{1}} (b_{i} ) \right) &= \sum_{ i = 1}^{ \ell } \varkappa_{n, p_{i}^{s_{i} } , 5}^{e_{1}} \left(  h_{p_{i}^{s_{i}}, 4, e_{1} }^{1, p_{i}^{s_{i} } , out } (b_{i}) \right) = \sum_{ i = 1}^{\ell} \varkappa_{n, p_{i}^{s_{i} } , 5}^{e_{1} } ( y_{i} ).
\end{align*}

By Diagram (\ref{d2}) in Theorem \ref{err:idealkthy}, $\rho_{n, 5 }^{e_{1} } ( a ) = h_{n, 4, e_{1}}^{1, n, out } ( h_{n, 4, e_{1 } }^{1, 1, in } ( a ) )$.  Therefore, 
\begin{align*}
 h_{n, 4, e_{1}}^{1, n, out } \left(  \sum_{ i = 1}^{\ell} \chi_{n, p_{i}^{s_{i} } , 4 }^{e_{1}} (b_{i} ) + h_{n, 4, e_{1 } }^{1, 1, in } ( a ) \right) 
 		= \sum_{ i = 1}^{\ell} \varkappa_{n, p_{i}^{s_{i} } , 5}^{e_{1} } ( y_{i} ) + \rho_{n, 5 }^{ e_{1} } ( a ) = x = h_{n, 4, e_{1}}^{1, n, out } ( z ).
\end{align*}
Hence, by Theorem \ref{err:idealkthy}, there exists $v \in F_{1, 0}^{e_{1}}$ such that $z =  \sum_{ i = 1}^{\ell} \chi_{n, p_{i}^{s_{i} } , 4 }^{e_{1}} (b_{i} ) + h_{n, 4, e_{1 } }^{1, 1, in } ( a ) + h_{n, 4, e_{1}}^{1, n, in} ( v )$.  Since $\alpha$ is zero on $H_{p_{i}^{s_{i}}, 4}^{e_{1}}$, $F_{1, 0}^{e_{1}}$, and $F_{1, 5}^{e_{1}}$, we have that $\alpha_{\e{n}{4} } ( z ) = 0$.

\bigskip

\emph{Claim 2:}  If $\alpha_{\mathfrak{e}_{k,4}}= 0$ for all $k = p_{1}^{s_{1}} \cdots p_{\ell}^{s_{\ell}}$, then $\alpha_{\mathfrak{e}_{n,4}} = 0$ for all $n\in\N_{\geq 2}$.

\bigskip

Suppose $n = p_{1}^{s_{1}} \cdots p_{\ell}^{s_{\ell}} m$, where $\gcd ( m, p_{i} ) = 1$.  Set $k = p_{1}^{s_{1}} \cdots p_{\ell}^{s_{\ell}}$.  Let $z \in H_{n,4}^{e_{1}}$.  Set $x = h_{n, 4, e_{1}}^{1,n, out} (z) \in F_{n, 5}^{e_{1}}$ and let $\overline{x} = \beta_{n, 5 }^{ e_{1}} (x) \in F_{1, 2}^{e_{1}}$.  Since $n \overline{x} = 0$ and since $\gcd ( m , p_{i} ) = 1$, $k \overline{x} = 0$.  Using a similar argument as in Claim 1 we get that $z = h_{n, 4, e_{1}}^{1,1, in } ( a ) + h_{n, 4, e_{1} }^{1, n, in} ( v ) + \chi_{n, k, 4}^{e_{1}} ( y )$ for some $v \in F_{1,0}^{e_{1}}$, $a \in F_{1, 5}^{e_{1}}$, and $y \in H_{k, 4}^{e_{1}}$.  Since $\alpha$ is zero on $F_{1,0}^{e_{1}}$, $F_{1, 5}^{e_{1}}$, and $H_{k, 4}^{e_{1}}$, we have that $\alpha_{\mathfrak{e}_{n,4}} ( z ) = 0$.

\bigskip

\emph{Claim 3:}  $\alpha_{ \mathfrak{e}_{n,4} } = 0$ for all $n = p_{i}^{r}$.

\bigskip

Let $n=p_i^r$ for some $i=1,\ldots,\ell$ and some $r\in\N$. 
We can assume that $r\neq n_{i,j}$ for every $j=1,\ldots,k_i$. Let $t$ be such that $n_{i,t} < r < n_{i,t+1}$ (and $t=0$ if $r<n_{i, 1}$ and $t=k_i$ if $n_{i, k_i} < r$). 
Let $z \in H_{n, 4}^{e_{1}}$ and set $x = h_{n, 4, e_{1}}^{1, n, out} ( z ) \in F_{n, 5}^{e_{1}}$.  Note that there exists $y \in F_{1, 5}^{e_{1}}$ and $z_{1}, \dots, z_{k_{i}} \in F_{n, 5}^{e_1}$ such that $x = \rho_{n, 5 }^{ e_{1} } (y) + \sum_{ j = 1}^{k_{i}} z_{j}$ and $\beta_{n, 5 }^{ e_{1} } ( z_{j} ) \in F_{1, 2}^{e_{1}}$ is an element of $0 \oplus \cdots \oplus 0 \oplus ( \Z_{p_{i}^{n_{i, j }} } )^{m_{i, j}} \oplus 0 \oplus \cdots \oplus 0$.  Set $n_{j} = p_{i}^{n_{i,j}}$.  We will show that either $z_{j} \in \im ( \kappa_{n, n_{j}, 5}^{e_{1}} \circ h_{n_{j}, 4, e_{1} }^{1, n_{j} , out } ) + \im(\rho_{n, 5 }^{ e_{1}} )$ or $z_{j} \in \im ( \varkappa_{n, n_{j}, 5}^{e_{1}} \circ h_{n_{j}, 4, e_{1} }^{1, n_{j} , out } ) + \im(\rho_{n, 5 }^{ e_{1}} )$ depending on whether $r<n_{i,j}$ or $r > n_{i,j}$.

First note that $n \beta_{n, 5 }^{ e_{1} } (z_{j} ) = 0$.  Note that the $\ker( \ftn{ \times n }{ F_{1,2}^{e_{1}} }{ F_{1,2}^{e_{1}} } )$ is
\begin{equation*}
0 \oplus \cdots \oplus 0 \oplus (\mathbb{Z}_{p_i^{n_{i,1}}})^{m_{i,1}}\oplus\cdots\oplus (\mathbb{Z}_{p_i^{n_{i,t}}})^{m_{i,t}}\oplus \langle p^{n_{i,t+1}-r}_i\rangle^{m_{i,t+1}}\oplus \cdots\oplus \langle p^{n_{i,k_i}-r}_i\rangle^{m_{i,k_i}}\oplus 0\oplus\cdots\oplus 0.
\end{equation*}
Also, since $F_{1,3}^{e_{1}} = 0$, by Theorem \ref{err:idealkthy}, we have that $\ftn{ h_{k,4, e_{1}}^{1, k, out} }{ H_{k, 4}^{e_{1}} }{ F_{k,5}^{e_{1}} }$ is surjective.  Hence, $\im ( \kappa_{n, n_{j}, 5}^{e_{1}} \circ h_{n_{j}, 4, e_{1} }^{1, n_{j} , out } ) = \im ( \kappa_{n, n_{j}, 5}^{e_{1}} )$ and $\im ( \varkappa_{n, n_{j}, 5}^{e_{1}} \circ h_{n_{j}, 4, e_{1} }^{1, n_{j} , out } ) = \im ( \varkappa_{n, n_{j}, 5}^{e_{1}} )$, when $n<n_j$ respectively $n>n_j$.

Suppose $r < n_{i, j}$.  Then $\beta_{n, 5 }^{ e_{1} } ( z_{j} ) = p_{i}^{n_{i, j}-r} g_{j}$ for some $g_{j} \in 0 \oplus \cdots \oplus 0 \oplus ( \Z_{p_{i}^{n_{i, j }} } )^{m_{i, j}} \oplus 0 \oplus \cdots \oplus 0$.  Since $p_{i}^{n_{i,j} } g_{j} = 0$, by Theorem \ref{err:idealkthy}, there exists $d_{j} \in F_{n_{j}, 5 }^{e_{1}}$ such that $\beta_{n_{j}, 5 }^{ e_{1} } ( d_{j} ) = g_{j}$.  By Lemma \ref{l:bmaps}, 
\begin{equation*}
\beta_{n, 5 }^{ e_{1} } ( \kappa_{n, n_{j} , 5 }^{e_{1}} ( d_{j} ) ) = p_{i}^{n_{i,j} - r } \beta_{n_{j}, 5 }^{ e_{1} } ( d_{j} ) = p_{i}^{n_{i,j} - r }  g_{j} = \beta_{n, 5 }^{ e_{1} } (z_{j}).
\end{equation*}
Hence, $z_{j} -  \kappa_{n, n_{j} , 5 }^{e_{1}} ( d_{j} )  \in \ker ( \beta_{n,5 }^{ e_{1} } ) = \im ( \rho_{n, 5 }^{ e_{1} } )$.  

Suppose $r > n_{i, j}$.  Then $n_{j} \beta_{n, 5 }^{ e_{1} } ( z_{j} ) = 0$.  By Theorem \ref{err:idealkthy}, there exists $d_{j} \in F_{n_{j}, 5 }^{e_{1}}$ such that $\beta_{n_{j} , 5 }^{ e_{1} } ( d_{j} ) =  \beta_{n, 5 }^{ e_{1} } ( z_{j} )$.  By Lemma \ref{l:bmaps},
\begin{equation*}
\beta_{n, 5 }^{ e_{1} } ( \varkappa^{e_{1}}_{n, n_{j} , 5} ( d_{j} ) ) = \beta_{n_{j} , 5 }^{ e_{1} } ( d_{j} ) = \beta_{n, 5 }^{ e_{1} } ( z_{j} ).
\end{equation*}
Hence, $z_{j} - \varkappa^{e_{1}}_{n, n_{j} , 5} ( d_{j} ) \in \ker ( \beta_{n,5 }^{ e_{1} } ) = \im ( \rho_{n, 5 }^{ e_{1} } )$.

By Lemma \ref{l:bmaps}, $h_{n, 4, e_{1}}^{1, n, out} \circ \omega_{n, n_{j}, 4}^{e_{1}} = \kappa_{n, n_{j} , 5}^{e_{1}} \circ h_{n_{j}, 4, e_{1}}^{1, n_{j}, out }$ when $r < n_{i,j}$ and $\varkappa_{n, n_{j} , 5 } \circ h_{n_{j}, 4, e_{1}}^{ 1, n_{j}, out } = h_{n, 4, e_{1}}^{1, n, out} \circ \chi_{n, n_{j}, 4}^{e_{1} }$ when $r\geq n_{i,j}$.  Consequently,
\begin{equation*}
h^{1,n,out}_{n,4,e_1}(z) = x \in \im ( \rho_{n, 5 }^{ e_{1} } ) + \sum_{ j = t+1}^{k_{i}} \im( h_{n, 4, e_{1}}^{1, n, out} \circ \omega_{n, n_{j}, 4}^{e_{1}} ) + \sum_{ j = 1}^{t} \im ( h_{n, 4, e_{1}}^{1, n, out} \circ \chi_{n, n_{j}, 4}^{e_{1} } ). 
\end{equation*}
By Diagram (\ref{d2}) in Theorem \ref{err:idealkthy}, $h^{1,n,out}_{n,4,e_1}\circ h^{1,1,in}_{n,4,e_1}=\rho_{n,5}^{e_1}$.  Hence,  
\begin{align*}
z &\in \im ( h_{n,4, e_{1}}^{1,1, in} ) + \sum_{ j = t+1}^{k_{i}} \im( \omega_{n, n_{j}, 4}^{e_{1}} ) + \sum_{ j = 1}^{t} \im ( \chi_{n, n_{j}, 4}^{e_{1} } ) + \ker( h_{n, 4, e_{1} }^{1,n, out } ) \\
	&\qquad\qquad =  \im ( h_{n,4, e_{1}}^{1,1, in} ) + \sum_{ j = t+1}^{k_{i}} \im( \omega_{n, n_{j}, 4}^{e_{1}} ) + \sum_{ j = 1}^{t} \im ( \chi_{n, n_{j}, 4}^{e_{1} } ) + \im( h_{n, 4, e_{1}}^{1, n, in } ).
\end{align*}
Therefore, $\alpha_{\e{n}{4}} ( z ) = 0$ since $\alpha$ is zero on $F_{1,0}^{e_{1}}$, $F_{1,5}^{e_{1}}$, and $H_{n_{j}, 4}^{e_{1}}$, for all $j=1,2,\ldots,k_i$.
\end{proof}

\begin{lemma}\label{l:red2}
Let $e_2 \colon \mathfrak{B}_{0} \hookrightarrow \mathfrak{B}_{1} \twoheadrightarrow \mathfrak{B}_{2}$ be an extension such that $K_{1} ( \mathfrak{B}_{2} )$ is torsion free and $K_{1} ( \mathfrak{B}_{0} ) =0$.  If $\alpha$ is an element $\Hom_{ \Lambda } ( \underline{K}_{\mathcal{E}} ( e_{1} ), \underline{K}_{\mathcal{E}} ( e_{2} ) )$ such that $\alpha_{\mathfrak{f}_{1,k}}=0$ and $\alpha_{\mathfrak{e}_{n,4}}=0$ for all $n\in\N_{\geq 2}$ and for all $k=0,1, \dots, 5$, then $\alpha$ is zero.  Consequently, if $\alpha \in \Hom_{ \Lambda } ( \underline{K}_{\mathcal{E} } ( \mathsf{S}e_{1} ) , \underline{K}_{\mathcal{E}} ( \mathsf{S} e_{2} ) )$, then $\alpha = 0$ if $\alpha_{ \e{n}{1} } = 0$ for all $n\in\N_{\geq 2}$ and $\alpha_{ \F{1}{k} } = 0$ for all $k=0,1,\dots, 5$.
\end{lemma}

\begin{proof}
Suppose $\alpha$ satisfies the assumption of the first part of the lemma.  Since $F_{1, 4}^{e_{1}}$ and $F_{1,5}^{e_{1}}$ are torsion free and since $F_{1,3}^{e_{1}} = 0$, by Theorem \ref{err:idealkthy}, $\ftn{ \rho_{n, i }^{ e_{1} } }{ F_{1, i}^{e_{1}} }{ F_{n, i}^{e_{1}} }$ is surjective for $i = 0, 1, 2$, and $\ftn{ h_{n,0, e_1}^{1,1, in} }{ F_{1, 1}^{e_{1}} }{ H_{n,0}^{e_{1}} }$, $\ftn{ h_{n, 1, e_{1} }^{n,1, in } }{ F_{n,1}^{e_{1}} }{ H_{n,1}^{e_{1}} }$, and $\ftn{ h_{n,4, e_{1} }^{1,n, out } }{ H_{n,4}^{e_{1} } }{ F_{n,5}^{e_{1}} }$ are surjective homomorphisms. Hence, $\alpha$ is zero on $F_{n, 0}^{e_{1}}$, $F_{n,1}^{e_{1}}$, $F_{n,2}^{e_{1}}$, $F_{n,5}^{e_{1}}$, $H_{n,0}^{e_{1}}$, and $H_{n,1}^{e_{1}}$ since $\alpha$ is zero on $F_{1,0}^{e_1}$, $F_{1,1}^{e_1}$, $F_{1, 2}^{e_1}$, and $H_{n,4}^{e_{1}}$.

Since $F_{1, 3}^{e_{1}} = F_{1, 3}^{e_{2}} = 0$, by Theorem \ref{err:idealkthy}, $h_{n, 2, e_{i}}^{1,1,out}$, $\beta_{n, 3 }^{ e_{i} }$, $h_{n, 4, e_{i} }^{n,1, in}$ are injective homomorphisms.  Let $x \in H_{n,2}^{e_{1}}$, $a \in F_{n,3}^{e_{1} }$, and $b \in F_{n,4}^{e_{1}}$.  Then 
\begin{align*}
h_{n,2, e_{2} }^{1,1,out} ( \alpha_{ \e{n}{ 2} } (x) ) &= \alpha_{ \F{1}{ 5 } } ( h_{n,2, e_{1} }^{1,1,out} (x) ) =0, \\
\beta_{n, 3 }^{ e_{2} } ( \alpha_{ \F{n}{ 3 } } (a) )&= \alpha_{ \F{1}{ 0 }  } ( \beta_{n, 3 }^{ e_{1} } (a) ) =0, \\
h_{n,4, e_{2} }^{n,1, in } ( \alpha_{ \F{n}{ 4 } } ( b ) ) &= \alpha_{\e{n}{ 4} }( h_{n,4, e_{1} }^{n,1, in }(b) ) =0. 
\end{align*}  
Hence, $\alpha_{ \e{n}{ 2}  } (x) =0$, $\alpha_{ \F{n}{ 3 }  } (a) = 0$, and $\alpha_{ \F{n}{ 4 }  } ( b ) = 0$.  Thus, $\alpha$ is zero on $H_{n,2}^{e_{1}}$, $F_{n,3}^{e_{1}}$, and $F_{n,4}^{e_{1}}$.

We are left with showing that $\alpha$ is zero on $H_{n,3}^{e_{1}}$ and $H_{n,5}^{e_{1}}$.  We first show that $\alpha$ is zero on $H_{n,5}^{e_{1}}$.  Let $x \in H_{n,5}^{e_{1}}$.  Since $F_{1,3}^{e_{1}} = 0$, by Theorem \ref{err:idealkthy}, there exists $y \in F_{1,1}^{e_{1}}$ such that $\tilde{f}_{1,1}^{e_1} ( y) = h_{n, 5, e_{1}}^{ 1, 1, out } ( x )$.  By Diagram (\ref{d2}) in Theorem \ref{err:idealkthy}, 
\begin{equation*}
h_{n, 5, e_{1} }^{1,1,out} ( h_{n,5, e_{1}}^{1,n,in} ( y ) ) = \tilde{f}_{1,1}^{e_1} ( y) = h_{n, 5, e_{1}}^{ 1, 1, out } ( x ).
\end{equation*}
So, $x \in \im (  h_{n,5, e_{1}}^{1,n,in} ) + \ker( h_{n, 5, e_{1}}^{ 1, 1, out } ) = \im (  h_{n,5, e_{1}}^{1,n,in} ) + \im( h_{n,5, e_{1}}^{1,1,in} )$.  Since $h_{n,5,e_{1}}^{1,n,in}$ and $h_{n,5,e_{1}}^{1,1,in}$ are functions on $F_{1,1}^{e_{1}}$ and $F_{1,0}^{e_{1}}$ respectively and since $\alpha$ is zero on $F_{1,1}^{e_{1}}$ and $F_{1,0}^{e_{1}}$, we have that that $\alpha_{ \e{n}{ 5}  } ( x ) = 0$.  Thus, $\alpha$ is zero on $H_{n,5}^{e_{1}}$.

Let $x \in H_{n,3}^{e_{1}}$.  First note that 
\begin{align*}
h_{n,3, e_{2} }^{n,1, out } ( \alpha_{ \e{ n }{ 3 } } (x) ) &= \alpha_{ \F{1}{ 5 } } ( h_{n,3, e_{1} }^{n,1, out } ( x ) ) = 0, \\
h_{n,3, e_{2} }^{1,1, out } ( \alpha_{ \e{ n }{ 3 }  } (x) ) &= \alpha _{ \F{1}{ 0 }  } ( h_{n,3, e_{1} }^{1,1, out } (x) ) = 0.
\end{align*}
Hence, by Theorem \ref{err:idealkthy}, there exists $y \in F_{1,4}^{e_{2}}$ such that $h_{n,3, e_{2}}^{1,1,in} (y) = \alpha_{ \e{ n }{ 3 }  } (x)$.  By Diagram (\ref{d1}) in Theorem \ref{err:idealkthy},
\begin{align*}
\tilde{f}_{1, 4}^{ e_{2} } (y) = h_{n,3, e_{2}}^{n,1, out} ( h_{n,3, e_{2}}^{1,1,in} (y) ) =  h_{n,3, e_{2}}^{n,1, out} ( \alpha_{ \e{ n }{ 3 } } (x) ) = 0.
\end{align*}
Since $F_{1,3}^{e_{2}} = 0$, by Theorem \ref{err:idealkthy} $\ftn{ \tilde{f}_{1, 4}^{ e_{2} } }{ F_{1,4}^{e_{2}} }{ F_{1, 5}^{e_{2}} }$ is injective.  Hence, $y =0$.  Thus, $\alpha_{\e{ n }{ 3 }  } (x) = 0$.  Therefore, $\alpha$ is zero on $H_{n,3}^{e_{1}}$.  
\end{proof}

\begin{theor}\label{t:inj}
Let $e_{2} \colon \mathfrak{B}_{0} \hookrightarrow \mathfrak{B}_{1} \twoheadrightarrow \mathfrak{B}_{2}$ be an extension of separable $C^{*}$-algebras 
with $K_{1} ( \mathfrak{B}_{0} ) = 0$ and $K_{1} ( \mathfrak{B}_{2} )$ torsion free.  Then the sequence
\begin{equation*}
\Hom_{ \Lambda} ( \underline{K}_{ \mathcal{E} } ( e ' ) , \underline{K}_{ \mathcal{E} } ( \mathsf{S} e_{2} ) ) \overset{ \psi^{*} }{ \rightarrow } \Hom_{ \Lambda } ( \underline{K}_{\mathcal{E} } ( e'' ) , \underline{K}_{\mathcal{E} } ( \mathsf{S} e_{2} ) ) \overset{ \phi^{*} }{ \rightarrow } \Hom_{ \Lambda } ( \underline{K}_{\mathcal{E} } ( \mathsf{S}e_{1} ) , \underline{K}_{ \mathcal{E} } (\mathsf{S}e_{2} ) ) 
\end{equation*}
induced by the exact sequence $\mathsf{S} e_{1} \overset{ \phi }{ \hookrightarrow } e'' \overset{\psi}{ \twoheadrightarrow} e'$ is exact. 
\end{theor}

\begin{proof}
A computation shows that the above sequence is a chain complex.  Let $\alpha \in \Hom_{ \Lambda } ( \underline{K}_{\mathcal{E}} ( e'' ), \underline{K}_{\mathcal{E}} ( \mathsf{S} e_{2} ) )$ such that $\alpha \circ  \kE ( \phi ) = 0$.  We want to construct $\beta \in \Hom_{\Lambda} ( \underline{K}_{\mathcal{E} } ( e' ), \underline{K}_{\mathcal{E}} (\mathsf{S}e_{2} ) )$ such that $ \beta \circ \kE ( \psi ) = \alpha$.

We will construct $\ftn{ ( \beta_{i} )_{ i = 0}^{5} }{ \ksix ( e') }{ \ksix( \mathsf{S} e_{2} ) }$ such that $\beta_{i} \circ \kE ( \psi )_{ \F{1}{i}  } = \alpha_{ \F{1}{ i } }$.  Then by Lemma \ref{l:freekthy}, there exists a unique $\beta \in \Hom_{\Lambda} ( \underline{K}_{\mathcal{E} } ( e' ), \underline{K}_{\mathcal{E}} (\mathsf{S}e_{2} ) )$ such that $ \beta  \circ \kE( \psi ) = \alpha$.  Recall that $\kE ( \psi )_{ \F{1}{i} } = \lambda_{i}$ in the projective resolution of $\ksix ( e_{1} )$.   

For convenience, we set $\alpha_{i} =   \alpha_{ \F{1}{ i } }$.  We must construct $\ftn{ \beta_{i} }{ F_{1,i}^{e'} }{ F_{1, i}^{\mathsf{S}e_{2} } }$ such that the diagrams
\begin{align*}
\xymatrix{
0 \ar[r]^{ \lambda_{3} } \ar[d]^{\cong} & 0 \ar[d]^{ \cong } 
&
0 \ar[r]^{ \lambda_{4} } \ar[d]^{\cong} & G_{4} \ar[d]^{ \cong } 
&
H_{5} \ar[r]^-{ \lambda_{5} } \ar[d]^{\cong} & G_{4} \oplus G_{5} \ar[d]^{ \cong } 
\\
F_{1,3}^{e''} \ar[r]^{ \lambda_{3} } \ar[d]_{\alpha_{3} } & F_{1,3}^{e'} \ar[ld]^{ \beta_{3} } 
&
F_{1,4}^{e''} \ar[r]^{ \lambda_{4} } \ar[d]_{\alpha_{4} } & F_{1,4}^{e'} \ar[ld]^{ \beta_{4} } 
&
F_{1,5}^{e''} \ar[r]^{ \lambda_{5} } \ar[d]_{\alpha_{5} } & F_{1,5}^{e'} \ar[ld]^{ \beta_{5} } 
\\
F_{1,3}^{\mathsf{S}e_{2} } 
&&
F_{1,4}^{\mathsf{S}e_{2} }
&&
F_{1,5}^{\mathsf{S}e_{2} }
 \\
H_{0} \ar[r]^-{ \lambda_{0} } \ar[d]^{\cong} & G_{5} \oplus \widetilde{F}_{0} \ar[d]^{ \cong } 
&
H_{1} \ar[r]^-{ \lambda_{1} } \ar[d]^{\cong} &  \widetilde{F}_{0}  \oplus F_{2} \ar[d]^{ \cong } 
&
H_{2} \ar[r]^-{ \lambda_{2} } \ar[d]^{\cong} & F_{2} \ar[d]^{ \cong } 
\\
F_{1,0}^{e''} \ar[r]^{ \lambda_{0} } \ar[d]_{\alpha_{0} } & F_{1,0}^{e'} \ar[ld]^{ \beta_{0} } 
&
F_{1,1}^{e''} \ar[r]^{ \lambda_{1} } \ar[d]_{\alpha_{1} } & F_{1,1}^{e'} \ar[ld]^{ \beta_{1} } 
&
F_{1,2}^{e''} \ar[r]^{ \lambda_{2} } \ar[d]_{\alpha_{2} } & F_{1,2}^{e'} \ar[ld]^{ \beta_{2} } 
\\
F_{1,0}^{\mathsf{S}e_{2} } 
&&
F_{1,1}^{\mathsf{S}e_{2} }
&&
F_{1,2}^{\mathsf{S}e_{2} }
}\end{align*}
are commutative.  It is clear from this that $\alpha_{3} = 0$, $\alpha_{4} = 0$.  Since $F_{1,0}^{\mathsf{S}e_{2}} \cong F_{1, 3}^{e_{2}} = 0$, $\alpha_{0} = 0$.  Also, we must necessarily have that $\beta_{3} = 0$ and $\beta_{0} = 0$.  So, it is only the third, fifth, and sixth diagram we need to check for commutativity.

\emph{Definition of $\beta_{2}$:}  Since $F_{1,2}^{e'}$ is free, we will define $\beta_{2}$ on its canonical generators.  Let $i \in \{ 1, \dots, \ell \}$ and $j \in \{ 1, \dots, k_{\ell} \}$.  Let $x$ be one of the canonical generators of $\Z^{ m_{i,j} } \subseteq F_{1,2}^{e'}$ and let $\tilde{x}$ be the corresponding canonical generator of $\Z^{m_{i,j} } \subseteq F_{1,2}^{e''}$.  

Set $k_{i,j} = p_{i}^{n_{i,j}}$.  We claim that $h_{ k_{i,j} , 1, e'' }^{1,1,in} ( \tilde{x} ) \in \im (\kE ( \phi )_{ \e{ k_{i,j}  }{ 1 } } )$.  First note that 
\begin{equation*}
\xymatrix{
H_{ k_{i,j},1}^{ \mathsf{S}e_{1} } \ar[rr]^-{  \kE ( \phi )_{ \e{ k_{i,j} }{  1 } } } & & H_{k_{i,j}, 1}^{e''} \ar[rr]^{ \kE ( \psi )_{ \e{ k_{i,j}  }{ 1 } } } & & H_{k_{i,j},1}^{e'}
}
\end{equation*}
is exact (\cf\ Remark~\ref{r:exact}).  Note that $\kE ( \psi )_{\e{ k_{i,j} }{ 1} } ( h_{ k_{i,j}, 1, e''}^{1,1,in} ( \tilde{x} ) ) = h_{ k_{i,j} , 1, e' }^{1,1,in} ( \lambda_{2}( \tilde{x} ) ) = h_{ k_{i,j} , 1, e' }^{1,1,in} ( k_{i,j} x )$.  By Lemma \ref{l:kken1}, $h_{ k_{i,j}, 1, e' }^{1,1,in} (k_{i,j} x ) = 0$.  Hence, $h_{ k_{i,j} , 1, e'' }^{1,1,in} ( \tilde{x} ) \in \im (\kE ( \phi )_{ \e{ k_{i,j} }{ 1 } } )$.  Since $\alpha \circ \kE(\phi) = 0$, we have that $\alpha_{ \e{ k_{i,j} } { 1 }  } ( h_{ k_{i,j}, 1, e'' }^{1,1,in} ( \tilde{x} ) ) = 0$.  Since 
\begin{equation*}
h_{ k_{i,j} , 1, \mathsf{S}e_{2} }^{1,1,in} ( \alpha_{ 2  } ( \tilde{x} ) ) = \alpha_{ \e{ k_{i,j} }{ 1 }  } ( h_{ k_{i,j}, 1, e'' }^{1,1,in} ( \tilde{x} ) ) = 0,
\end{equation*}  
by Theorem \ref{err:idealkthy}, there exists $z_{ \tilde{x} } \in F_{1,1}^{ \mathsf{S} e_{2}}$ such that $k_{i,j} f_{1, 1}^{ \mathsf{S}e_{2} } ( z_{ \tilde{x} } ) =  \alpha_{ 2  } ( \tilde{x} )$.  Set $\beta_{2} ( x ) = f_{1, 1}^{ \mathsf{S}e_{2} } ( z_{ \tilde{x} } )$.  We do this for all the generators of $\Z^{m_{i,j}}$ and for all $i,j$.  Moreover, we let $\beta_{2}$ be zero on $\Z^{s}$.  By construction, we have that $\beta_{2} \circ \lambda_{2} = \alpha_{2}$. 

\emph{Definition of $\beta_{1}$:}  Recall that $F_{1,1}^{e'} = F_{1} = \widetilde{F}_{0} \oplus F_{2}$.  We define $\ftn{ \beta_{1} }{ F_{1,1}^{e'} }{ F_{1,1}^{\mathsf{S}e_{2}} }$ by setting it to zero on $\widetilde{F}_{0}$ and on $\Z^{s} \subseteq F_{2}$.  We now define $\beta_{1}$ on the rest of the canonical generators.  Let $i \in \{ 1,\dots, \ell \}$ and $j \in \{ 1, \dots, k_{\ell} \}$ be given.  We want to define $\beta_{1}$ on the generators of $\Z^{m_{i,j}}$.  Let $x$ be one of the canonical generators of $\Z^{m_{i,j}}$.  Define $\beta_{1} ( x ) = z_{ \tilde{x} }$ where $z_{\tilde{x}}$ is defined as above.  By construction, $f_{1,1}^{\mathsf{S} e_{2} } \circ \beta_{1} = \beta_{2} \circ f_{1,1}^{ e'}$.  Since $F_{1,0}^{\mathsf{S}e_{2} } = 0$, we have that $f_{1,1}^{ \mathsf{S} e_{2}}$ is injective and since
\begin{align*}
f_{1,1}^{ \mathsf{S} e_{2}} \circ \beta_{1} \circ \lambda_{1} &= \beta_{2} \circ f_{1,1}^{ e' } \circ \lambda_{1} \\
				&= \beta_{2} \circ \lambda_{2} \circ f_{1,1}^{ e''} \\
				&= \alpha_{2} \circ f_{1,1}^{e''} \\
				&= f_{1,1}^{\mathsf{S}e_{2} } \circ \alpha_{1}
\end{align*} 
we have that $ \beta_{1} \circ \lambda_{1} = \alpha_{1}$.

\emph{Definition of $\beta_{0}$:} As mentioned above, we need to have $\beta_{0} = 0$.

\emph{Definition of $\beta_{5}$:}  Since $\mathsf{S} \mathfrak{A}_{2} \hookrightarrow\mathfrak{A}_2''\twoheadrightarrow \mathfrak{A}_{2}'$ is an extension of separable, nuclear $C^{*}$-algebras in the bootstrap category $\mathcal{N}$ with all the $K$-theory being finitely generated, the sequence
 $$\xymatrix{
    \Hom_{\Lambda}(\underline{K}(\mathfrak{A}_2'),\underline{K}(\mathsf{S}\mathfrak{B}_2))\ar[r]^{\psi^*}&
    \Hom_{\Lambda}(\underline{K}(\mathfrak{A}_2''),\underline{K}(\mathsf{S}\mathfrak{B}_2))\ar[r]^{\phi^*}&
    \Hom_{\Lambda}(\underline{K}(\mathsf{S}\mathfrak{A}_2),\underline{K}(\mathsf{S}\mathfrak{B}_2))\ar[d]\\
    \Hom_{\Lambda}(\underline{K}(\mathfrak{A}_2),\underline{K}(\mathsf{S}\mathfrak{B}_2))\ar[u]&
    \Hom_{\Lambda}(\underline{K}(\mathsf{S}\mathfrak{A}_2''),\underline{K}(\mathsf{S}\mathfrak{B}_2))\ar[l]^{\phi^*}&
    \Hom_{\Lambda}(\underline{K}(\mathsf{S}\mathfrak{A}_2'),\underline{K}(\mathsf{S}\mathfrak{B}_2))\ar[l]^{\psi^*}
    }$$
  is exact (using the UMCT of D\u{a}d\u{a}rlat and Loring, \cf\ \cite{multcoeff}).  Since $\alpha \circ \underline{K} ( \phi ) = 0$, there exists $\eta \in \Hom_{ \Lambda } ( \underline{K} ( \mathfrak{A}_{2}' ), \underline{K} ( \mathsf{S} \mathfrak{B}_{2} ) )$ such that $ \eta \circ \underline{K} ( \psi )  = \alpha$ on $\underline{K} ( \mathfrak{A}_{2} '' )$.  Set $\ftn{ \beta_{5} = \eta \vert_{ F_{1,5 }^{e'} } }{ F_{1,5}^{ e'} }{ F_{1,5}^{ \mathsf{S} e_{2} } }$.  Consequently, $\beta_{5} \circ \lambda_{5} = \alpha_{5}$.
  
\emph{Definition $\beta_{4}$:}  Since $F_{1,4}^{e'} = F_{1,4}^{e_{1}}$ is free and $f_{1,4}^{\mathsf{S}e_{2} }$ is surjective, there exists a homomorphism $\ftn{ \beta_{4} }{ F_{1,4}^{e'} }{ F_{1,4}^{\mathsf{S}e_{2}} }$ such that $f_{1,4}^{ \mathsf{S}e_{2} } \circ \beta_{4} = \beta_{5} \circ f_{1,4}^{ e'}$.  

\emph{Definition $\beta_{3}$:}  As mentioned above, we need to have $\beta_{3} = 0$. 

By construction, $\beta_{i} \circ \lambda_{i} = \alpha_{i}$ for $i = 0, 1, \dots, 5$ and $( \beta_{i} )_{ i = 0}^{5}$ is a homomorphism from $\ksix ( e' )$ to $\ksix ( \mathsf{S} e_{2} )$.
\end{proof}

\begin{theor}\label{t:surj}
Let $e_{2} \colon \mathfrak{B}_{0} \hookrightarrow \mathfrak{B}_{1} \twoheadrightarrow \mathfrak{B}_{2}$ be an extension of separable $C^{*}$-algebras with $K_{1} ( \mathfrak{B}_{0} ) = 0$ and $K_{1} ( \mathfrak{B}_{2} )$ torsion free.  Then the sequence
\begin{equation*}
\Hom_{ \Lambda} ( \underline{K}_{ \mathcal{E} } ( e '' ) , \underline{K}_{ \mathcal{E} } ( \mathsf{S} e_{2} ) ) \overset{ \phi^{*} }{ \rightarrow } \Hom_{ \Lambda } ( \underline{K}_{\mathcal{E} } ( \mathsf{S}e_{1} ) , \underline{K}_{\mathcal{E} } ( \mathsf{S} e_{2} ) ) \overset{ \delta }{ \rightarrow } \Hom_{ \Lambda } ( \underline{K}_{\mathcal{E} } ( \mathsf{S}e' ) , \underline{K}_{ \mathcal{E} } ( \mathsf{S}e_{2} ) ) 
\end{equation*}
induced by the exact sequence $\mathsf{S} e_{1} \overset{ \phi }{ \hookrightarrow } e'' \overset{\psi}{ \twoheadrightarrow} e'$ is exact. 
\end{theor}

\begin{proof}
It is easy to check that the above sequence is a chain complex.  

Let $\alpha \in \Hom_{ \Lambda } ( \underline{K}_{\mathcal{E} } ( \mathsf{S}e_{1} ) , \underline{K}_{\mathcal{E} } ( \mathsf{S} e_{2} ) )$ such that $\delta ( \alpha ) = 0$.  By Lemma \ref{l:freekthy}, if $( \beta_{i} )_{ i = 0 }^{5}$ is a homomorphism from $\ksix ( e'' )$ to $\ksix ( \mathsf{S} e_{2} )$, then there exists a unique $\beta \in \Hom_{ \Lambda } ( \underline{K}_{\mathcal{E} } (e'' ) , \underline{K}_{\mathcal{E} } ( \mathsf{S} e_{2} ) )$ such that $\beta_{ \F{1}{i} }  = \beta_{i}$.  By Lemma \ref{l:red1} and Lemma \ref{l:red2}, it is enough to construct $( \beta_{i} )_{ i = 0}^{5}$ such that its unique lifting $\beta \in \Hom_{ \Lambda } ( \underline{K}_{\mathcal{E} } (e'' ) , \underline{K}_{\mathcal{E} } ( \mathsf{S} e_{2} ) )$ satisfies the property that $\alpha - \beta \circ \kE ( \phi )$ is zero on $H_{n,1}^{ \mathsf{S}e_{1} }$ and $F_{1, k }^{\mathsf{S}e_{1}}$ for all $n = p_{i}^{n_{i,j}}$ and $k=0,1, \dots, 5$.

Since the homomorphisms from $F_{1, k}^{e'}$ to $F_{1,k}^{\mathsf{S}^{2} e_{1 }}$ are surjective and $\delta ( \alpha )=0$, we have that $\alpha$ is zero on $F_{1,k}^{\mathsf{S} e_{1} }$.  Note also that any choice of chain map $( \beta_{i} )_{ i = 0 }^{5}$, its unique lifting $\beta$ satisfies the property that $ \beta \circ \kE ( \phi ) = 0$ on $F_{1, k }^{ \mathsf{S}e_{1} }$.  Consequently, we need a chain map $( \beta_{i} )_{ i = 0}^{5}$ from $\ksix ( e'' )$ to $\ksix ( \mathsf{S} e_{2} )$ such that its unique lifting $\beta$ satisfies the property that $ ( \beta \circ \kE ( \phi ) )_{ \e{n}{1} } =\alpha_{ \e{ n }{ 1 } } $ for $n = p_{i}^{n_{i,j}}$ for all $i = 1, \dots , \ell$ and $j =1, \dots, k_{\ell}$.  

First we want to define a homomorphism from $F_{1,2}^{e''}$ to $F_{1, 1 }^{ \mathsf{S}e_{2} }$.  It will be defined on each of the direct summands as follows.  Let $i \in\{ 1, \ldots, \ell \}$ and $j \in \{ 1, \dots, k_{i} \}$ and set $n = p_{i}^{n_{i,j}}$.  Recall that $F_{1,2}^{e''}$ is
\begin{align*}
  & \Z^{m_{1,1}}\oplus\cdots\oplus\Z^{m_{1,k_1}}  \\
  & \oplus\Z^{m_{2,1}}\oplus\cdots\oplus\Z^{m_{2,k_2}} \\
  & \oplus\cdots\oplus\Z^{m_{\ell,1}}\oplus\cdots\oplus\Z^{m_{\ell,k_{\ell}}}.
\end{align*}
Let $e_{i,j,t}$ be one of the canonical generator of $\Z^{m_{i,j}}$.  Note that 
\begin{equation*}
\kE ( \psi )_{ \e{ n }{ 1 } } ( h_{n,1, e''}^{1,1, in} (e_{i,j,t}) ) = h_{n,1, e'}^{1,1,in} ( \lambda_{2} (e_{i,j,t}) ) = h_{n,1, e'}^{1,1,in} ( n e_{i,j,t} ) 
\end{equation*}  
By Lemma \ref{l:kken1}, $h_{n,1, e'}^{1,1,in} ( n e_{i,j,t} ) = 0$.  Since 
\begin{equation*}
\xymatrix{ H_{n,1}^{\mathsf{S}e_{1} }  \ar[rr]^{ \kE ( \phi )_{ \e{ n }{ 1 } } }& & H_{n,1}^{e''} \ar[rr]^{ \kE ( \psi )_{ \e{ n}{ 1 } } } & & H_{n,1}^{e'}  }
\end{equation*}  
is exact (\cf\ Remark~\ref{r:exact}), there exists $x_{i,j, t } \in H_{n,1}^{\mathsf{S}e_{1}}$ such that $\kE ( \phi )_{ \e{ n } {1} } ( x_{i,j,t} ) = h_{n,1, e''}^{1,1, in} (e_{i,j,t})$.

Since $\alpha$ is zero on $F_{1,4}^{\mathsf{S}e_{1}}$ and $h_{n,1, \mathsf{S}e_{1} }^{1,1, out } ( x_{i,j,t } ) \in F_{1,4}^{\mathsf{S}e_{1}}$, 
\begin{align*}
h_{n, 1, \mathsf{S}e_{2} }^{1,1, out } ( \alpha_{ \e{ n }{ 1 } } ( x_{i, j , t } ) ) &=  \alpha_{\F{1}{4} } ( h_{n, 1, \mathsf{S}e_{1} }^{1,1, out } (x_{i,j,t} ) ) \\
			&= 0.
\end{align*}
By Theorem \ref{err:idealkthy}, there exists $y_{i,j,t} \in F_{1,2}^{ \mathsf{S}e_{2}}$ such that $h_{n,1, \mathsf{S}e_{2}}^{1,1,in} ( y_{i,j,t} ) =  \alpha_{ \e{ n }{ 1 }  } ( x_{i, j , t } )$.  
Note that 
\begin{align*}
h_{n,1, \mathsf{S}e_{2} }^{n,1, out } ( h_{n,1, \mathsf{S}e_{2}}^{1,1,in} ( y_{i,j,t} ) ) &= h_{n,1, \mathsf{S}e_{2} }^{n,1, out } ( \alpha_{ \e{ n }{ 1 } } ( x_{i, j , t } ) ) \\
			&= \alpha_{\F{1}{3} } ( h_{n,1, \mathsf{S}e_{1} }^{n,1, out } ( x_{i,j,t} ) ) \\
			&= 0.
\end{align*} 
By Diagram (\ref{d1}) in Theorem \ref{err:idealkthy},
$h_{n,1, \mathsf{S}e_{2} }^{n,1, out } ( h_{n,1, \mathsf{S}e_{2}}^{1,1,in} ( y_{i,j,t} ) ) = \tilde{f}_{ 1, 2}^{ \mathsf{S}e_{2}} ( y_{i,j,t} ) = 0$.  By Theorem \ref{err:idealkthy}, there exists $z_{i,j,t} \in F_{1,1}^{\mathsf{S}e_{2}}$ such that $\tilde{f}_{1,1}^{ \mathsf{S}e_{2} } ( z_{i,j,t} ) = y_{i,j,t}$.  Note that $\tilde{f}_{1,1}^{ \mathsf{S} e_{2} } = f_{1,1}^{ \mathsf{S}e_{2}}$.  Define $\ftn{ \zeta_{i,j} }{ \Z^{m_{i,j}} }{ F_{1,1}^{ \mathsf{S}e_{2} } }$ by $\zeta_{i,j} ( e_{i,j,t} ) = z_{i,j,t}$.  Define $\ftn{ \zeta }{ F_{1,2}^{ e'' } }{ F_{1,1}^{ \mathsf{S}e_{2} } }$ by 
\begin{equation*}
\zeta = \sum_{ i = 1}^{\ell} \sum_{ j = 1}^{k_{i} } \zeta_{i,j}.
\end{equation*}  

\emph{Definition of the chain map $( \beta_{i} )_{ i = 0}^{5}$:}  Define $\ftn{ \beta_{1} }{ F_{1,1}^{e''} }{ F_{1,1}^{ \mathsf{S}e_{2} } }$ by $\zeta \circ f_{1,1}^{e''}$.  Define $\ftn{ \beta_{2} }{ F_{1,2}^{e''} }{ F_{1,2}^{ \mathsf{S}e_{2}} }$ by $f_{1,1}^{ \mathsf{S}e_{2} } \circ \zeta$.  Set $\beta_{i} = 0$ for $i = 0, 3, 4, 5$.  A computation shows that $( \beta_{i} )_{ i = 0 }^{5}$ is a homomorphism from $\ksix ( e'' )$ to $\ksix ( \mathsf{S} e_{2} )$. 

Let $\beta \in \Hom_{ \Lambda } ( \underline{K}_{\mathcal{E} } (e'' ) , \underline{K}_{\mathcal{E} } ( \mathsf{S} e_{2} ) )$ be the unique homomorphism such that $\beta_{\F{1}{i} } = \beta_{i}$.  We need to check that $\alpha - \beta \circ \kE ( \phi ) = 0$ on $H_{n, 1}^{ \mathsf{S}e_{1} }$ for all $n = p_{i}^{n_{i,j}}$.  

Since $\delta ( \alpha ) = 0$ and $\delta ( \beta \circ \kE ( \phi )  ) = 0$, we have that $\alpha_{ \e{ n }{ 1 } } ( y ) =  ( \beta \circ \kE ( \phi ) )_{ \e{ n }{ 1} } (y) = 0$ for all $y \in \im ( H_{n,1}^{\mathsf{S}e'}  \rightarrow H_{n,1}^{ \mathsf{S} e_{1} } )$. 
 
Note that
\begin{align*}
(\beta\circ\underline{K}_{\mathcal{E}}(\phi))_{\mathfrak{e}_{p_i^{n_{i,j}},1}} ( x_{i,j,t} ) &= \beta ( h_{ p_{i}^{n_{i,j}}, 1, e'' }^{1,1, in } ( e_{i,j,t} ) ) \\
	&= h_{ p_{i}^{n_{i,j}}, 1, \mathsf{S} e_{2} }^{1,1, in } ( \beta ( e_{i,j,t} ) ) \\
	&= h_{ p_{i}^{n_{i,j}}, 1, \mathsf{S} e_{2} }^{1,1, in }  ( f_{1,1}^{ \mathsf{S} e_{2} } ( z_{i,j,t}) ) \\
	&= h_{ p_{i}^{n_{i,j}}, 1, \mathsf{S} e_{2} }^{1,1, in }  ( y_{i,j,t} ) \\
	&= \alpha_{\mathfrak{e}_{p_i^{n_{i,j}},1}} ( x_{i,j,t} ).
\end{align*}
 For each $i=1,\ldots,\ell$ and each $j=1,\ldots,k_i$, set
$$X_{i,j}= \operatorname*{span}_{1\leq t\leq m_{i,j}} x_{i,j,t}.$$
By the above equation, we have that
$$(\beta\circ\underline{K}_{\mathcal{E}}(\phi))_{\mathfrak{e}_{p_i^{n_{i,j}},1}}(x)=\alpha_{\mathfrak{e}_{p_i^{n_{i,j}},1}}(x),\quad\text{for all }x\in X_{i,j}.$$

For convenience, set $d_{i,j} = p_{i}^{n_{i,j}}$.  We will first show that $ ( \beta  \circ \kE ( \phi ) ) = \alpha$ on $H_{ d_{i,1} , 1}^{ \mathsf{S}e_{1}}$.  Let $x \in H_{d_{i,1} , 1}^{\mathsf{S}e_{1}}$.  Since $\kE ( \phi )_{ \e{ d_{i,1} }{  1 } } ( x ) \in \ker( \ftn{ \kE ( \psi )_{ \e{ d_{i,1} }{ 1 } } }{ H_{d_{i,1},1}^{e''} }{ H_{d_{i,1},1}^{e'} } )$, by Lemma \ref{l:kken1} there exist $z_{1}, z_{2} , \dots, z_{k_{i}}\in F_{1,2}^{e''}$ with $z_{j}$ in 
\begin{equation*}
0 \oplus \cdots \oplus 0 \oplus \Z^{m_{i,j}} \oplus 0 \oplus \cdots \oplus 0 \subseteq F_{1,2}^{e''}
\end{equation*}  
such that 
\begin{equation*}
\kE ( \phi )_{ \e{ d_{i,1} }{  1 } }(x) = \sum_{ j = 1}^{ k_{i} } h_{ d_{i,1} , 1, e'' }^{1,1, in } (z_{j}).
\end{equation*}

By Lemma \ref{l:bmaps}, $h_{d_{i,1}, 1, e'' }^{1,1, in } ( z_{j} ) = \omega^{e''}_{d_{i,1} , d_{i,j} , 1} ( h_{d_{i,j}, 1, e''}^{1,1, in } (z_{j}) )$. By Lemma \ref{l:kken1}, $h_{d_{i,j}, 1, e''}^{1,1, in } (z_{j})$ is in the subgroup of $H_{d_{i,j},1}^{e''}$ generated by $\setof{ h_{d_{i,j} , 1, e''}^{1,1,in} ( e_{i,j,t} ) }{ 1 \leq t \leq m_{i,j} }$.  Since $\underline{K}_{\mathcal{E}} (\phi )_{ \mathfrak{e}_{ d_{i,j} ,1 } }( x_{i,j,t} ) = h_{ d_{i,j} , 1, e''}^{1,1,in} ( e_{i,j,t} )$, there exists $x_{j} \in X_{i,j}\subseteq H_{d_{i,j} , 1}^{\mathsf{S}e_{1}}$ such that $\kE ( \phi )_{\e{ d_{i,j} }{ 1 } } ( x_{j} ) = h_{d_{i,j}, 1, e''}^{1,1, in } (z_{j} )$.

Set $y_{j} =  \omega_{ d_{i,1} , d_{i,j} , 1}^{ \mathsf{S}e_{1} } ( x_{j} )$.  Then 
\begin{align*}
\kE ( \phi )_{ \e{ d_{i,1} }{ 1 }  } ( y_{j} ) = \omega_{d_{i,1} , d_{i,j} , 1}^{ e'' } ( \kE ( \phi )_{ \e{ d_{i,j} }{ 1 } } ( x_{j} ) )= h_{d_{i,1}, 1, e'' }^{1,1, in } ( z_{j} ).
\end{align*}
Hence, $\kE ( \phi )_{ \e{ d_{i,1} }{ 1} } ( x ) = \kE ( \phi )_{ \e{ d_{i,1} }{ 1 } } \left(  \sum_{j=1}^{k_{i}}  y_{j} \right)$.  Hence, $x - \sum_{j=1}^{k_{i}}  y_{j}  \in \ker( \kE ( \phi )_{ \e{ d_{i,1} }{ 1 } } ) = \im ( H_{d_{i,1},1}^{\mathsf{S}e'}  \rightarrow H_{d_{i,1} , 1}^{ \mathsf{S} e_{1} } )$.  Thus, $\alpha_{ \e{ d_{i,1} } {1 } } (x ) = \alpha_{\e{ d_{i,1} }{ 1 } } \left( \sum_{j=1}^{k_{i}}  y_{j} \right)$ and $( \beta  \circ \kE ( \phi ) )_{ \e{ d_{i,1} }{ 1 } } (x) = ( \beta \circ \kE ( \phi ) )_{ \e{ d_{i,1} }{ 1 } } \left( \sum_{j=1}^{k_{i}}  y_{j} \right)$.  Note that 
\begin{align*}
( \beta \circ \kE ( \phi ) )_{ \e{ d_{i,1} }{ 1 }  }  ( y_{j} ) &= \beta_{ \e{ d_{i,1} }{ 1}  }( \underline{K}_{\mathcal{E}}(\phi)_{\e{ d_{i,1} }{ 1} } ( y_{j} ) ) \\
				&= \beta_{\e{ d_{i,1} }{ 1 } } ( \kE ( \phi )_{ \e{ d_{i,1} }{ 1} }  ( \omega_{ d_{i,1} , d_{i,j} , 1}^{ \mathsf{S}e_{1} } ( x_{j} ) ) ) \\
				&= \beta_{ \e{ d_{i,1} }{ 1 } } ( \omega_{d_{i,1}, d_{i,j} , 1}^{ e'' } ( \kE ( \phi )_{ \e{ d_{i,j} }{ 1}  } (x_{j}) ) ) \\
				&= \omega_{d_{i,1}, d_{i,j} , 1}^{ \mathsf{S}e_2 } ( \beta_{ \e{ d_{i,j} }{ 1} } ( \kE ( \phi )_{ \e{ d_{i,j} }{ 1 } } ( x_{j} ) ) ) \\
				&= \omega_{d_{i,1}, d_{i,j} , 1}^{ \mathsf{S}e_2 } ( ( \beta \circ \kE ( \phi) )_{\e{ d_{i,j} }{ 1 } } ( x_{j} ) ) \\
                                &= \omega_{d_{i,1}, d_{i,j} , 1}^{ \mathsf{S}e_2 } ( \alpha_{\mathfrak{e}_{d_{i,j},1}}(x_j) )\qquad\text{since }x_j\in X_{i,j}  \\
                                &= \alpha_{ \e{ d_{i,1} }{ 1 } } ( \omega_{d_{i,1}, d_{i,j} , 1}^{ \mathsf{S}e_{1} } (x_{j} ) ) \\
				&= \alpha_{\e{ d_{i,1} }{ 1} } ( y_{j} ).
\end{align*}
Therefore, $ ( \beta  \circ \kE( \phi ) )_{ \e{d_{i,1}}{1} }( x ) = \alpha_{ \e{ d_{i,1} }{ 1 } } ( x )$.  We have just shown that $ ( \beta \circ \kE ( \phi ) )_{ \e{ d_{i,1} }{ 1 } } = \alpha_{ \e{ d_{i,1} }{ 1 } }$.

Let $x \in H_{ d_{i,2} , 1}^{\mathsf{S}e_{1} }$.  Since $\kE ( \phi )_{ \e{ d_{i,2} }{ 1 } } ( x ) \in \ker( \ftn{ \kE ( \psi )_{ \e{ d_{i,2 } }{ 1 } } }{ H_{d_{i,2},1}^{e''} }{ H_{d_{i,2},1}^{e''} } )$, by Lemma \ref{l:kken1}, there exist $z_{1}, z_{2} , \dots, z_{k_{i}}\in F_{1,2}^{e''}$ with $z_{j}$ in 
\begin{equation*}
0 \oplus \cdots \oplus 0 \oplus \Z^{m_{i,j}} \oplus 0 \oplus \cdots \oplus 0 \subseteq F_{1,2}^{e''}
\end{equation*}  
such that 
\begin{equation*}
\kE ( \phi )_{ \e{ d_{i,2} }{ 1 } }(x) = \frac{ d_{i,2} }{ d_{i,1} }  h_{ d_{i,2} , 1, e'' }^{1,1, in } ( z_{1}) + \sum_{ j = 2}^{ k_{i} } h_{ d_{i,2} , 1, e'' }^{1,1, in } (z_{j}).
\end{equation*}
Arguing as in the case $d_{i,1}$, for $j \geq 2$, there exists $x_{j}$ in $H_{d_{i,j},1}^{ \mathsf{S} e_{1} }$ satisfying the following:
\begin{itemize}
\item[(1)] $ ( \beta  \circ \kE ( \phi ) ) _{ \e{ d_{i,2} }{ 1 } }( \omega_{d_{i,2}, d_{i,j} , 1}^{ \mathsf{S}e_{1} } ( x_{j} ) ) = \alpha_{ \e{ d_{i,2} }{ 1 }  } ( \omega_{d_{i,2}, d_{i,j} , 1}^{ \mathsf{S}e_{1} } ( x_{j} ) )$

\item[(2)] $\kE ( \phi )_{ \e{ d_{i,2} }{ 1 } } ( \omega_{ d_{i,2} , d_{i,j} ,1 }^{ \mathsf{S}e_{1} } ( x_{j} ) ) = \omega_{d_{i,2} , d_{i, j } , 1}^{e''} ( \kE ( \phi )_{ \e{ d_{i,j} }{ 1 } } ( x_{j} ) ) = h_{ d_{i,2} , 1, e'' }^{1,1, in } (z_{j})$.
\end{itemize}

By Lemma \ref{l:bmaps}, $ \chi_{ d_{i,2}, d_{i,1}, 1 }^{e''} \circ h_{d_{i,1} , 1, e''}^{1,1, in } = \frac{ d_{i,2} }{ d_{i,1} } h_{d_{i,2} , 1, e'' }^{1,1,in}$. By Lemma \ref{l:kken1}, $h_{d_{i,1}, 1, e''}^{1,1, in } (z_{1})$ is in the subgroup of $H_{d_{i,1},1}^{e''}$ generated by $\setof{ h_{d_{i,1} , 1, e''}^{1,1,in} ( e_{i,1,t} ) }{ 1 \leq t \leq m_{i,1} }$.  Since $\underline{K}_{\mathcal{E}} (\phi )_{ \mathfrak{e}_{ d_{i,1} ,1 } }( x_{i,1,t} ) = h_{ d_{i,1} , 1, e''}^{1,1,in} ( e_{i,1,t} )$, there exists $x_{1} \in X_{i,1}\subseteq H_{d_{i,1} , 1}^{\mathsf{S}e_{1}}$ such that $\kE ( \phi )_{\e{ d_{i,1} }{ 1 } } ( x_{1} ) = h_{d_{i,1}, 1, e''}^{1,1, in } (z_{1} )$. Note that 
\begin{align*}
\kE ( \phi )_{ \e{ d_{i,2} }{ 1 } } ( \chi_{d_{i,2} , d_{i,1}, 1}^{ \mathsf{S}e_{1} } ( x_{1} ) ) &= \chi_{ d_{i,2} , d_{i,1} , 1}^{e''} ( \kE( \phi )_{ \e{ d_{i,1} }{ 1} } ( x_{1} ) ) \\
				&= \chi_{ d_{i,2} , d_{i,1} , 1}^{e''} ( h_{d_{i,1} , 1, e''}^{1,1, in } (z_{1}) ) \\
				&= \frac{ d_{i,2} }{ d_{i,1} } h_{d_{i,2} , 1, e'' }^{1,1,in} ( z_{1} ) .
\end{align*}
Moreover, 
\begin{align*}
  ( \beta \circ \kE ( \phi)  )_{ \e{ d_{i,2}}{ 1 } }( \chi_{d_{i,2} , d_{i,1}, 1}^{ \mathsf{S}e_{1} } ( x_{1}) ) 
  & = \beta_{\mathfrak{e}_{d_{i,2},1}} ( \underline{K}_{\mathcal{E}}(\phi)_{\mathfrak{e}_{d_{i,2},1}}  (\chi_{d_{i,2} , d_{i,1}, 1}^{ \mathsf{S}e_1 }( x_{1} )))  \\
  & = \beta_{\mathfrak{e}_{d_{i,2},1}} ( \chi_{d_{i,2} , d_{i,1}, 1}^{ e'' } (\underline{K}_{\mathcal{E}}(\phi)_{\mathfrak{e}_{d_{i,1},1}}( x_{1} )))  \\
  & = \chi_{d_{i,2} , d_{i,1}, 1}^{ \mathsf{S}e_{2} } (\beta_{\mathfrak{e}_{d_{i,1},1}}(\underline{K}_{\mathcal{E}}(\phi)_{\mathfrak{e}_{d_{i,1},1}}( x_{1} )))  \\
  & = \chi_{d_{i,2} , d_{i,1}, 1}^{ \mathsf{S}e_{2} } ((\beta\circ\underline{K}_{\mathcal{E}}(\phi))_{\mathfrak{e}_{d_{i,1},1}}( x_{1} ))  \\
  & = \chi_{d_{i,2} , d_{i,1}, 1}^{ \mathsf{S}e_{2} } (\alpha_{ \e{ d_{i,1} }{ 1 } }( x_{1} )) \qquad\text{since }x_1\in X_{i,1} \\
  & = \alpha_{ \e{ d_{i,2} }{ 1 } } ( \chi_{d_{i,2} , d_{i,1}, 1}^{ \mathsf{S}e_{1} } ( x_{1}) ).
\end{align*}
Set $y_{1} = \chi_{d_{i,2} , d_{i,1}, 1}^{ \mathsf{S}e_{1} } ( x_{1})$ and $y_{j} = \omega_{ d_{i,2} , d_{i,j} , 1}^{ \mathsf{S}e_{1}} ( x_{j} )$ for all $j \geq 2$.  Then we have that $( \beta \circ \kE ( \phi)  )_{ \e{ d_{i,2}}{ 1 } }( y_{j} ) = \alpha_{ \e{ d_{i,2} }{ 1 } } ( y_{j} )$ for all $j = 1, 2, \dots, k_{i}$.  Also note that 
\begin{align*}
 \sum_{ j = 1}^{ k_{i} } \kE( \phi )_{ \e{ d_{i,2}}{ 1} } ( y_{j} ) = \kE ( \phi )_{ \e{ d_{i,2} }{ 1} } ( x ).
\end{align*}
Therefore, $x - \sum_{ j = 1}^{ k_{i} } y_{j}  \in \ker( \kE ( \phi )_{ \e{ d_{i,2} }{ 1 } } ) = \im ( H_{d_{i,2},1}^{\mathsf{S}e'}  \rightarrow H_{d_{i,2} , 1}^{ \mathsf{S} e_{1} } )$.  Hence, $\alpha_{ \e{ d_{i,2} }{ 1} } ( x ) = \sum_{ j = 1}^{ k_{i} }  \alpha_{ \e{ d_{i,2} }{ 1 } } ( y_{j} ) = \sum_{ j  =1}^{ k_{i} } ( \beta \circ \kE ( \phi ) )_{ \e{ d_{i,2} }{ 1 } }( y_{j} ) = ( \beta \circ \kE ( \phi ) )_{ \e{ d_{i,2} }{ 1 } }( x )$.  We have just shown that $\alpha_{ \e{ d_{i,2} }{  1 } }  = ( \beta \circ \kE ( \phi ) )_{ \e{ d_{i,2} }{  1 } }  $.  

We can continue this process to show that $\alpha_{ \e{d_{i,j} }{ 1} } =  ( \beta \circ \kE (\phi ) )_{ \e{d_{i,j} }{ 1} }$ for all $j  =1, 2, \dots, k_{i}$.  Hence, $\alpha_{ \e{d_{i,j} }{ 1} } =  ( \beta \circ \kE (\phi ) )_{ \e{d_{i,j} }{ 1} }$ for all $i = 1, 2, \dots, \ell$ and $j = 1, 2, \dots, k_{i}$.
\end{proof}

\section{Isomorphism theorem}
In this section we prove a UMCT for a certain class of extensions, \ie, we prove that the natural homomorphism $\Gamma_{e_1,e_2}$ from $\kkE(e_1,e_2)$ to $\Hom_{\Lambda}(\kE(e_1),\kE(e_2))$ is an isomorphism, for extensions $e_1$ and $e_2$ of separable, nuclear $C\sp*$-algebras in the bootstrap category $\mathcal{N}$ with $K$-groups of the associated cyclic six term exact sequences being finitely generated, zero exponential map and with the $K_1$-groups of the quotients being free abelian groups.

\begin{lemma}[General 5-lemma]
  Let there be given a commutative diagram
  $$\xymatrix{
    A\ar[d]^\alpha\ar[r]&
    B\ar[d]^\beta\ar[r]&
    C\ar[d]^\gamma\ar[r]&
    D\ar[d]^\delta\ar[r]&
    E\ar[d]^\epsilon \\
    A'\ar[r]&
    B'\ar[r]&
    C'\ar[r]&
    D'\ar[r]&
    E' }
  $$
  with abelian groups and group homomorphisms, where both rows are
  complexes. Then the following holds:
  
  \begin{enumerate}
  \item 
    If the first row is exact at $C$, the second row is exact at $B'$,
    and $\alpha$ is surjective, and $\beta$ and $\delta$ are
    injective, then $\gamma$ is injective.
  \item
    If the first row is exact at $D$, the second row is exact at $C'$,
    and $\beta$ and $\delta$ are surjective, and $\epsilon$ is
    injective, then $\gamma$ is surjective.
  \end{enumerate}
\end{lemma}
\begin{proof}
  Diagram chase. 
\end{proof}

\begin{propo}\label{prop:injckalgs}
  Let $e_1\colon \mathfrak{A}_0 \hookrightarrow \mathfrak{A}_1 \twoheadrightarrow \mathfrak{A}_2$ and
  $e_2\colon \mathfrak{B}_0 \hookrightarrow \mathfrak{B}_1  \twoheadrightarrow\mathfrak{B}_2$
  be extensions of separable, nuclear $C^{*}$-algebras in
  in the bootstrap category $\mathcal{N}$, with the associated cyclic six term exact sequences in $K$-theory being finitely generated, $K_1(\mathfrak{A}_0)=K_1(\mathfrak{B}_0)=0$,
  and $K_1(\mathfrak{A}_2)$ and $K_1(\mathfrak{B}_2)$ being torsion free. 

  Then $\Gamma_{e_1,e_2}$ is an isomorphism. 
\end{propo}

\begin{proof}
  Let moreover $e'$ and $e''$ be as in the previous section. 
  Since $\Gamma_{ - , -}$ is natural,
  \begin{equation*}
    \def\objectstyle{\scriptstyle}
    \def\labelstyle{\scriptstyle}  
    \xymatrix{
      \kkE(e',\mathsf{S} e_2)\ar[r]\ar[d]_{\Gamma_{e',\mathsf{S}e_2}} 
      & \kkE(e'',\mathsf{S}e_2)\ar[r]\ar[d]_{\Gamma_{e'',\mathsf{S}e_2}} 
      & \kkE(\mathsf{S}e_1,\mathsf{S}e_2)\ar[d]_{\Gamma_{\mathsf{S}e_1,\mathsf{S}e_2}}\ar[r] 
      & \\
      \Hom_{\Lambda}(\kE(e'),\kE(\mathsf{S}e_2))\ar[r] 
      & \Hom_{\Lambda}(\kE(e''),\kE(\mathsf{S}e_2))\ar[r] 
      & \Hom_{\Lambda}(\kE(\mathsf{S}e_1),\kE(\mathsf{S}e_2))\ar[r] 
      & \\
      \ar[r] 
      & \kkE(\mathsf{S}e',\mathsf{S}e_2)\ar[d]_{\Gamma_{\mathsf{S}e',\mathsf{S}e_2}} \ar[r]
      & \kkE(\mathsf{S}e'',\mathsf{S}e_2)\ar[d]_{\Gamma_{\mathsf{S}e'',\mathsf{S}e_2}} 
      \\
      \ar[r]
      & \Hom_{\Lambda}(\kE(\mathsf{S}e'),\kE(\mathsf{S}e_2))\ar[r]
      & \Hom_{\Lambda}(\kE(\mathsf{S}e''),\kE(\mathsf{S}e_2)) 
      }
  \end{equation*}
  is commutative, the top row is exact, and by Theorems~\ref{t:inj}
  and~\ref{t:surj}, 
  \begin{equation*}
    \def\objectstyle{\scriptstyle}
    \def\labelstyle{\scriptstyle}  
    \xymatrix{
      \Hom_{\Lambda}(\kE(e'),\kE(\mathsf{S}e_2))\ar[r] 
      & \Hom_{\Lambda}(\kE(e''),\kE(\mathsf{S}e_2))\ar[r] 
      & \Hom_{\Lambda}(\kE(\mathsf{S}e_1),\kE(\mathsf{S}e_2))\ar[r] 
      & \Hom_{\Lambda}(\kE(\mathsf{S}e'),\kE(\mathsf{S}e_2)) } 
  \end{equation*}
  is also exact. 
  
  Since $\ksix(e')$ and $\ksix(e'')$ are projective,
  $\Gamma_{e',\mathsf{S}e_2}$, $\Gamma_{e'',\mathsf{S}e_2}$, 
  $\Gamma_{\mathsf{S}e',\mathsf{S}e_2}$, and  
  $\Gamma_{\mathsf{S}e'',\mathsf{S}e_2}$ are isomorphisms 
  (\cf\ Lemma~\ref{l:freekthy}).  
  So by the general 5-lemma (see above), $\Gamma_{\mathsf{S}e_1,\mathsf{S}e_2}$ is an isomorphism. 
  By Lemma \ref{l:suspension}, $\Gamma_{e_1,e_2}$ is an isomorphism.
\end{proof}

\begin{theor}\label{thm:main}
  Let $e_1\colon\mathfrak{A}_0 \hookrightarrow \mathfrak{A}_1 \twoheadrightarrow \mathfrak{A}_2$ and
  $e_2\colon \mathfrak{B}_0 \hookrightarrow \mathfrak{B}_1 \twoheadrightarrow \mathfrak{B}_2$ 
  be extensions of separable, nuclear $C^{*}$-algebras in the bootstrap category $\mathcal{N}$ with the associated cyclic six term exact sequences in $K$-theory being finitely generated, zero exponential map, and $K_1(\mathfrak{A}_2)$ and
  $K_1(\mathfrak{B}_2)$ being torsion free. 
  
  Then the natural map 
  $\ftn{\Gamma_{e_1,e_2}}{\kkE(e_1,e_2)}{\Hom_{\Lambda}(\kE(e_1),\kE(e_2))}$ 
  is an isomorphism. 
\end{theor}

\begin{proof}
  Note that for each $j=1,2$, we have that $\ksix(e_j)\cong\ksix(e_{j,1})\oplus\ksix(e_{j,0})$ where
  $e_{j,i}$ is an extension of separable, nuclear $C^{*}$-algebras in
  the bootstrap category $\mathcal{N}$ such that $\ksix(e_{j,1})$ is of the form 
  \begin{equation*}
    \xymatrix{ 
      \ast \ar[r] & \ast \ar[r] &  \ast \ar[d] \\
      \ast \ar[u] & \ast \ar[l] & 0 \ar[l]  }
  \end{equation*}
  and $\ksix(e_{j,0})$ is of the form 
  \begin{equation*}
    \xymatrix{ 
      0 \ar[r] & 0 \ar[r] &  0 \ar[d] \\
      0\ar[u] & \ast \ar[l] & \ast \ar[l]
      }
  \end{equation*}
  where the $K_1$'s of $\ksix(e_{j,1})$ is finitely generated and free, and the $K_{1}$'s of $\ksix(e_{j,0})$ are finitely
  generated.

By Bonkat's UCT and since $\Gamma_{-,-}$ is natural, we have that the diagrams
\begin{equation*}
    \xymatrix{
      \kkE(e_1,e_2)\ar[r]^{\cong}\ar[d]_{\Gamma_{e_1,e_2}}
      & \kkE(e_{1,1}\oplus e_{1,0},e_2)
      \ar[d]^{\Gamma_{e_{1,1}\oplus e_{1,0},e_2}}  \\
      \Hom_{\Lambda}(\kE(e_1),\kE(e_2)\ar[r]_-{\cong} 
      & \Hom_{\Lambda}(\kE(e_{1,1}\oplus e_{1,0}),\kE(e_2)) , }
  \end{equation*}
  \begin{equation*}
    \xymatrix{
      \kkE(e_{1,1}\oplus e_{1,0},e_2)\ar[r]^{\cong}\ar[d]_{\Gamma_{e_{1,1}\oplus e_{1,0},e_2}} 
      & \kkE(e_{1,1}\oplus e_{1,0},e_{2,1}\oplus e_{2,0})\ar[d]^{\Gamma_{e_{1,1}\oplus e_{1,0},e_{2,1}\oplus e_{2,0}}} \\
      \Hom_{\Lambda }(\kE(e_{1,1}\oplus e_{1,0}),\kE(e_2))\ar[r]_-{\cong} 
      & \Hom_{\Lambda}(\kE(e_{1,1}\oplus e_{1,0}),\kE(e_{2,1}\oplus e_{2,0}))  }
  \end{equation*}
  are commutative.
    
  Hence, to prove that $\Gamma_{e_1,e_2}$ is an isomorphism it is enough to
  prove that  
  $\Gamma_{e_{1,1}\oplus e_{1,0 },e_{2,1}\oplus e_{2,0}}$ is an isomorphism.  
  Since $\Gamma_{e_{1,0},e_{2,1}\oplus e_{2,0}}$ and
  $\Gamma_{e_{1,1},e_{2,0}}$ are isomorphisms (this follows from Lemma~\ref{l:freekthy} and an argument using \cite[Lemma~7.1.5]{bonkat} and the UMCT of D\u{a}d\u{a}rlat and Loring, \cf\ \cite{multcoeff}), it is enough to prove
  that $\Gamma_{e_{1,1},e_{2,1}}$ is an isomorphism.

  Since $e_{1,1}$ is $\kkE$-equivalent to an extension as in 
  Proposition~\ref{prop:injckalgs}, by the naturality of 
  $\Gamma_{-,-}$ and by Proposition~\ref{prop:injckalgs}, 
  $\Gamma_{e_{1,1},e_{2,1}}$ is an isomorphism.
\end{proof}

\begin{remar}
The class of extensions that the above theorem applies to contains the following classes: 
\begin{itemize}
\item
The class of all Cuntz-Krieger algebras satisfying condition (II) of Cuntz with one specified ideal.
\item 
The class of all purely infinite Cuntz-Krieger algebras with one specified ideal.
\item 
The class of all Cuntz-Krieger algebras with one specified gauge invariant ideal.
\item 
The class of all graph algebras satisfying condition (K) with one specified ideal and finitely generated $\ksix$. 
\item 
The class of all unital graph algebras satisfying condition (K) with one specified ideal. 
\item 
The class of all purely infinite graph algebras with one specified ideal and finitely generated $\ksix$.
\item 
The class of all graph algebras with one specified gauge invariant ideal and finitely generated $\ksix$.
\end{itemize}
\end{remar}

\begin{remar}\label{rem-remaboutscopeofinvariant}
We do not know of any counterexample to a general UMCT for the ideal-related $K$-theory with coefficients, neither do we know how to prove a general UMCT with one specified ideal. 
But it is clear that we can get the analogue result when we rotate the conditions on the $K$-theory. We also get the analogue result when either of the variables has four zero groups in the associated cyclic six term exact sequence. Moreover, it is possible do certain direct sums of these as well. 
\end{remar}

\section{Reduced invariant}

\begin{remar}
  By going over the proof once again, it is clear that if we use the invariant consisting of all the groups $F_{1,i}^e$, $F_{n,i}^e$, $H_{n,i}^e$ for all $n\in\N_{\geq 2}$ and $i=0,1,2,3,4,5$, and all the homomorphisms $f_{1,i}^e$, $f_{n,i}^e$, $\beta_{n,i}^e$, $\rho_{n,i}^e$, $\kappa_{n,mn,i}^e$, $\varkappa_{mn,m,i}^e$, $\omega_{n,mn,i}^e$, $\chi_{mn,m,i}^e$, $h_{n,i,e}^{1,1,in}$, $h_{n,i,e}^{1,1,out}$, $h_{n,i,e}^{1,n,in}$, $h_{n,i,e}^{1,n,out}$, $h_{n,i,e}^{n,1,in}$, $h_{n,i,e}^{n,1,out}$ for all $m,n\in\N_{\geq 2}$ and $i=0,1,2,3,4,5$, then the corresponding theorem still holds, \ie, Theorem~\ref{thm:main} still holds if we replace $\Hom_{\Lambda}(\kE(e_1),\kE(e_2))$ with the homomorphisms between the groups only respecting those natural transformations mentioned above. 
\end{remar}

\begin{remar}
  We do not know yet whether these are all the homomorphisms to include. Let $e_1 \colon \mathfrak{A}_{0} \hookrightarrow \mathfrak{A}_{1} \twoheadrightarrow \mathfrak{A}_{2}$ and $e_2 \colon \mathfrak{B}_{0} \hookrightarrow \mathfrak{B}_{1} \twoheadrightarrow \mathfrak{B}_{2}$ be extensions of separable \cstar-algebras, with $e_1$ in the bootstrap category $\mathcal{N}$. Then we have -- as noted in \cite[Remark~6.6]{err:idealkthy} -- natural homomorphisms $G_i\colon \kkE(e_1,e_2)\rightarrow \kk(\A_i,\B_i)$, for $i=0,1,2$. Using \cite[Remark~6.6]{err:idealkthy}, we see that $G_0$ is an isomorphism if $K_0(\A_2)=K_1(\A_2)=0$, $G_1$ is an isomorphism if $K_0(\A_0)=K_1(\A_0)=0$, $G_2$ is an isomorphism if $K_0(\A_1)=K_1(\A_1)=0$. Using this, the UCT of Rosenberg and Schochet and the UCT of Bonkat, it is not so hard to see that $\kkE(e,e')$ for $e,e'\in\{\mathfrak{f}_{n,i}:n\in\N_{\geq 2},i=0,1,2,3,4,5\}$ are generated by the $\kkE$-classes of the maps mentioned above. We believe that it is true also when one of the entries is of the form $\mathfrak{e}_{n,i}$ and the proof should be similar as well. For $\kkE(\mathfrak{e}_{n,i},\mathfrak{e}_{m,j})$ Example~9.1 of \cite{err:idealkthy} indicates that there might be some natural homomorphisms we have not described yet. 
\end{remar}

\begin{defin}
  We now define a reduced invariant, the \bolddefine{reduced ideal-related $K$-theory with coefficients} $\kE^\textnormal{red}(e)$, which consists of all the groups $F_{1,i}^e$, $F_{n,i}^e$, $H_{n,4}^e$ for all $n\in\N_{\geq 2}$ and $i=0,1,2,3,4,5$. A homomorphism between such invariants should be a family of group homomorphisms between these groups which respects the natural homomorphisms  $f_{1,i}^e$, $f_{n,i}^e$, $\beta_{n,i}^e$, $\rho_{n,i}^e$, $\kappa_{n,mn,i}^e$, $\varkappa_{mn,m,i}^e$, $\omega_{n,mn,4}^e$, $\chi_{mn,m,4}^e$, $h_{n,4,e}^{1,1,in}$, $h_{n,4,e}^{1,1,out}$, $h_{n,4,e}^{1,n,in}$, $h_{n,4,e}^{1,n,out}$, $h_{n,4,e}^{n,1,in}$, $h_{n,4,e}^{n,1,out}$ for all $m,n\in\N_{\geq 2}$ and $i=0,1,2,3,4,5$.
\end{defin}

\begin{theor}\label{thm:mainreduced}
  Let $e_1\colon\mathfrak{A}_0 \hookrightarrow \mathfrak{A}_1 \twoheadrightarrow \mathfrak{A}_2$ and
  $e_2\colon \mathfrak{B}_0 \hookrightarrow \mathfrak{B}_1 \twoheadrightarrow \mathfrak{B}_2$ 
  be extensions of separable, nuclear $C^{*}$-algebras in the bootstrap category $\mathcal{N}$ with the associated cyclic six term exact sequences in $K$-theory being finitely generated, zero exponential map, and $K_1(\mathfrak{A}_2)$ and
  $K_1(\mathfrak{B}_2)$ being torsion free.  
  
  Then the natural map 
  $\ftn{\Gamma_{e_1,e_2}^\textnormal{red}}{\kkE(e_1,e_2)}{\Hom_{\Lambda^{\textnormal{red}}}(\kE^\textnormal{red}(e_1),\kE^\textnormal{red}(e_2))}$ 
  is an isomorphism. 
\end{theor}
\begin{proof}
  By going through the whole proof once again, we see that this is all we need in order to prove the theorem. 
\end{proof}

\begin{remar}
  There are of course a few informations in the invariant which are uniquely determined -- \eg the homomorphism from $K_0(\A_2)$ to $K_0(\B_2)$ is uniquely determined by the rest of the invariant. But it would make the definition of a reduced invariant quite fragmented if we were to take all such things out of the invariant. 
\end{remar}

\section{Automorphisms of Cuntz-Krieger algebras}

In this section, we use our invariant $\kE^\textnormal{red}$ to classify the automorphism group of a stabilized Cuntz-Krieger algebra with exactly one non-trivial ideal up to equivalences: unitary homotopy equivalence and approximate unitary equivalence.

\subsection{Unitary homotopy equivalence}

\begin{defin}
  Let $\A$ and $\B$ be \cstar-algebras. We say that two \starhoms $\varphi,\psi\colon \A\rightarrow \multialg{ \B }$ are \bolddefine{unitarily homotopic} if there exists a norm-continuous map $[0,\infty)\ni t\mapsto u_t\in U( \multialg{ \B })$ such that for $a\in\A$ and $t\in[0,\infty)$ we have that
$$\varphi(a)-u_t^*\psi(a)u_t\in\B$$
and, moreover, 
$$\lim_{t\rightarrow\infty}u_t^*\psi(a)u_t=\varphi(a),$$
for all $a\in\A$. 
If the images of $\varphi$ and $\psi$ lie inside $\B$, the first
condition is superfluous.

For each \cstar-algebra \A we let $\operatorname{\overline{Inn}}_h(\A)$ denote the set of automorphisms of $\A$, which are unitarily homotopic to the identity morphism on $\A$. 
\end{defin}

Kirchberg has a very general classification theorem for (strongly) purely infinite \cstar-algebras in \cite[Folgerung~4.3]{kirchpure}, which in the case we are looking at specializes to the following.

\begin{theor}[Kirchberg]
  Let $e\colon\mathfrak{A}_0 \hookrightarrow \mathfrak{A}_1 \twoheadrightarrow \mathfrak{A}_2$ 
  be an essential extension of separable, nuclear, stable, simple, purely infinite $C^{*}$-algebras in the bootstrap category $\mathcal{N}$. Then we have a short exact sequence of groups
  $$\xymatrix{\{1\}\ar[r] & \operatorname{\overline{Inn}}_h(\A_1)\ar[r] & \operatorname{Aut}(\A_1)\ar[r] & \kkE^{-1}(e,e)\ar[r] & \{1\},}$$
 where the group operation on $\operatorname{Aut}(\A_1)$ and $\kkE( e, e)^{-1}$ are composition of maps and Kasparov product, respectively. 
\end{theor}

When we combine this theorem with the UMCT from previous sections we get the following corollary. 

\begin{corol}\label{c:unithomotopy}
  Let $e\colon\mathfrak{A}_0 \hookrightarrow \mathfrak{A}_1 \twoheadrightarrow \mathfrak{A}_2$ 
  be an essential extension of separable, nuclear, stable, simple, purely infinite $C^{*}$-algebras in the bootstrap category $\mathcal{N}$ with all groups in the associated cyclic six term exact sequence being finitely
  generated, zero exponential map, and $K_1(\mathfrak{A}_2)$ being torsion free. 
  Then we have a short exact sequence of groups
  $$\xymatrix{\{1\}\ar[r] & \operatorname{\overline{Inn}}_h(\A_1)\ar[r] & \operatorname{Aut}(\A_1)\ar[r] & \operatorname{Aut}_{\Lambda^{\textnormal{red}}}(\kE^\textnormal{red}(e))\ar[r] & \{1\}.}$$
\end{corol}

\begin{remar}
We can apply the above corollary to the following $C^{*}$-algebras:
\begin{itemize}
\item[(1)] $\A_1 = \B_{1} \otimes \K$, where $\B_{1}$ is a Cuntz-Krieger algebra satisfying condition (II) of Cuntz with exactly one non-trivial ideal. 

\item[(2)]  $\A_1 = \B_{1} \otimes \K$, where $\B_{1}$ is a purely infinite graph algebra with exactly one non-trivial ideal and finitely generated $K$-theory.  
\end{itemize} 
\end{remar}

\subsection{Approximately unitary equivalence}

\begin{defin}
  Let $\A$ and $\B$ be \cstar-algebras. We say that two \starhoms $\varphi,\psi\colon \A\rightarrow \multialg{ \B }$ are \bolddefine{approximately unitarily equivalent} if there exists a sequence of elements $\{ u_{n} \}_{ n = 1}^{\infty}$ in $U( \multialg{ \B } )$ such that for $a\in\A$ and $n \in \N$ we have that
$$\varphi(a)-u_n^*\psi(a)u_n\in\B$$
and, moreover, 
$$\lim_{n\rightarrow\infty}u_n^*\psi(a)u_n=\varphi(a),$$
for all $a\in\A$. 
If the images of $\varphi$ and $\psi$ lie inside $\B$, the first
condition is superfluous.

For each \cstar-algebra \A we let $\operatorname{\overline{Inn}}(\A)$ denote the set of automorphisms of $\A$, which are approximately unitarily equivalent to the identity morphism on $\A$. 
\end{defin}

\begin{theor}\label{t:aue}
Let $e_{1} \colon \mathfrak{A}_{0} \hookrightarrow \mathfrak{A}_{1} \twoheadrightarrow \mathfrak{A}_{2}$ and $e_{2} \colon \mathfrak{B}_{0} \hookrightarrow \mathfrak{B}_{1} \twoheadrightarrow \mathfrak{B}_{2}$ be extensions of separable $C^{*}$-algebras.  Let $\ftn{ \phi, \psi }{ e_{1} }{ e_{2} }$ be homomorphisms.  If $\mathfrak{A}_{1}$ is semiprojective, $\mathfrak{A}_{0}$ is generated by projections, and $\phi_{1}$ and $\psi_{1}$ are approximately unitarily equivalent, then $\kE ( \phi ) = \kE ( \psi )$.
\end{theor}

\begin{proof}
We identify $\mathfrak{A}_{0}$ with its image in $\mathfrak{A}_{1}$ and $\mathfrak{B}_{0}$ with its image in $\mathfrak{B}_{1}$.  With this identification, $\phi_{0}$ becomes the restriction $\phi_{1} \vert_{ \mathfrak{A}_{0}}$.  Since $\mathfrak{A}_{1}$ is semiprojective, by Proposition 2.3 of \cite{mdge_onepara}, there exists a unitary $u \in \multialg{ \mathfrak{B}_{1} }$ such that $\mathrm{ad} ( u ) \circ  \phi_{1}$ is homotopic to $\psi_{1}$.  Let $\ftn{ \Phi }{ \mathfrak{A}_{1} }{ C( [ 0 , 1 ] , \mathfrak{B}_{1} ) }$ be a homotopy from $\mathrm{ad} ( u ) \circ  \phi_{1}$ to $\psi_{1}$. 

We claim that $\Phi (a) \in C( [ 0 ,1 ] , \mathfrak{B}_{0} )$ for all $a \in \mathfrak{A}_{0}$.  Since $\mathfrak{A}_{0}$ is generated by projections, it is enough to show that $\mathrm{ev}_{t} \circ \Phi (p) \in \mathfrak{B}_{0}$ for all projections $p \in \mathfrak{A}_{0}$ and for all $t \in [0,1]$.  Let $p$ be a projection in $\mathfrak{A}_{0}$.  Then there exists a partition $0 = t_{0} < t_{1} < \dots < t_{n-1} < t_{n} = 1$ of $[0,1]$ such that $\| \mathrm{ev}_{t} \circ \Phi(p) - \mathrm{ev}_{s} \circ \Phi(p) \| < 1$ for all $t,s \in [ t_{i-1} , t_{i} ]$.  In particular, $\| \mathrm{ev}_{0} \circ \Phi(p)- \mathrm{ev}_{t} \circ \Phi(p) \| < 1$ for all $t \in [ 0 , t_{1} ]$.  Therefore, $\mathrm{ev}_{t} \circ \Phi(p)$ is Murray-von Neumann equivalent to $\mathrm{ev}_{0} \circ \Phi(p) = \mathrm{ad} (u) \circ \phi_{1} (p)$ for all $t \in [ 0 , t_{1} ]$.  Since $\phi$ is a homomorphism from $e_{1}$ to $e_{2}$, we have that $\mathrm{ad}(u) \circ \phi_{1} (p) \in \mathfrak{B}_{0}$.  Since $\mathfrak{B}_{0}$ is an ideal of $\mathfrak{B}_{1}$, we have that $\mathrm{ev}_{t} \circ \Phi (p) \in \mathfrak{B}_{0}$ for all $t \in [ 0 , t_{1} ]$.  Using the same argument in the interval $[t_{1}, t_{2}]$, we get that $\mathrm{ev}_{t} \circ \Phi (p) \in \mathfrak{B}_{0}$ for all $t \in [ t_{1} , t_{2} ]$.  Continuing this process, we get that $\mathrm{ev}_{t} \circ \Phi (p) \in \mathfrak{B}_{0}$ for all $t \in [0,1]$.  Hence, $\Phi (p ) \in C( [0,1] , \mathfrak{B}_{0} )$.  We have just proved our claim.

We have shown that $\mathrm{ad}(u) \circ \phi_{1}$ and $\psi_{1}$ are homotopic with the homotopy respecting the canonical ideals $\mathfrak{A}_{0}$ and $\mathfrak{B}_{0}$.  Thus, $\kkE ( \mathrm{ad}(u) \circ \phi ) = \kkE( \psi )$.  Since $\kkE ( \mathrm{ad}(u) \circ \phi ) = \kkE( \phi )$, we have that $\kkE( \phi ) = \kkE(\psi)$.  Therefore, $\kE ( \phi ) = \kE ( \psi )$.
\end{proof}

We believe that Theorem~\ref{t:aue} is true without the semiprojectivity assumption and the assumption that the ideal $\mathfrak{A}_{0}$ is generated by projections but we have not been able to obtain a proof.

\begin{corol}\label{c:approx-unit}
Let $e\colon\mathfrak{A}_0 \hookrightarrow \mathfrak{A}_1 \twoheadrightarrow \mathfrak{A}_2$ be an essential extension of separable, nuclear, stable, simple purely infinite $C^{*}$-algebras in the bootstrap category $\mathcal{N}$ with finitely
  generated $K$-theory, zero exponential map, and $K_1(\mathfrak{A}_2)$ being torsion free.  If $\mathfrak{A}_{1}$ semiprojective, then we have a short exact sequence of groups
  $$\xymatrix{\{1\}\ar[r] & \operatorname{\overline{Inn}}(\A_1)\ar[r] & \operatorname{Aut}(\A_1)\ar[r] & \operatorname{Aut}_{\Lambda^{\textnormal{red}}}(\kE^\textnormal{red}(e))\ar[r] & \{1\}.}$$
   \end{corol}

\begin{proof}
By Theorem~\ref{t:aue}, $\xymatrix{
\operatorname{\overline{Inn}}(\A_1)\ar[r] & \operatorname{Aut}(\A_1)\ar[r] & \operatorname{Aut}_{\Lambda^{\textnormal{red}}}(\kE^\textnormal{red}(e))
}$ the chain complex is an exact sequence.  It is clear that $\xymatrix{
\{ 1 \} \ar[r] & \operatorname{\overline{Inn}}(\A_1)\ar[r] & \operatorname{Aut}(\A_1)
}$ is an exact sequence.  By Corollary~\ref{c:unithomotopy}, $\xymatrix{
\operatorname{Aut}(\A_1)\ar[r] & \operatorname{Aut}_{\Lambda^{\textnormal{red}}}(\kE^\textnormal{red}(e))\ar[r] & \{1\}
}$ is exact.  
\end{proof}

\begin{remar}
We can apply the above corollary to the following $C^{*}$-algebras:
\begin{itemize}
\item[(1)] $\A_1 = \B_{1} \otimes \K$, where $\B_{1}$ is a Cuntz-Krieger algebra satisfying condition (II) of Cuntz with exactly one non-trivial ideal. 

\item[(2)]  $\A_1 = \B_{1} \otimes \K$, where $\B_{1}$ is a unital purely infinite graph algebra with exactly one non-trivial ideal.  The fact that $\A_{1}$ is semiprojective follows from the results in \cite{ek_graph}.
\end{itemize} 
\end{remar}


\section*{Acknowledgement}

Supported by the Danish National Research Foundation through the Centre for Symmetry and Deformation (DNRF92).

\end{document}